\documentclass[12pt,english]{article}
\usepackage{amsfonts,babel,amssymb,amsmath,amsthm}
\usepackage{graphicx}
\usepackage{framed}
\usepackage{mathabx}
\usepackage[shortlabels]{enumitem}

\newcommand{\mc}{\mathcal}
\newcommand{\mr}[1]{{\mathrm{#1}}}
\newcommand{\ul}{\underline}
\newcommand{\ds}{\displaystyle}
\newcommand{\punto}{{\,\cdot\,}}
\newcommand{\floor}[1]{\lfloor #1 \rfloor}
\newcommand{\ceil}[1]{\lceil #1 \rceil}
\renewcommand{\Re}{\mathrm{Re}\,}

\newcommand{\dd}{\mathrm d}
\newcommand{\R}{\mathbb R}
\newcommand{\C}{\mathbb C}
\newtheorem{proposition}{Proposition}[section]
\newtheorem{corollary}[proposition]{Corollary}
\newtheorem{lemma}[proposition]{Lemma}
\newtheorem{theorem}[proposition]{Theorem}

\newtheorem*{hypothesis*}{Hypothesis}

\numberwithin{equation}{section}

\title{Brushing up a theorem by Lehel Banjai \\ on the convergence of \\
Trapezoidal Rule Convolution Quadrature\footnote{Both authors are partially supported by NSF grant DMS 1818867}}

\author{Hasan Eruslu\footnote{Department of Mathematical Sciences, University of Delaware. {\tt heruslu@udel.edu}}, Francisco-Javier Sayas\footnote{Department of Mathematical Sciences, University of Delaware. {\tt fjsayas@udel.edu}} \\
University of Delaware}

\date{\today}

\begin{document}

\maketitle

\begin{abstract}
This document is made up of two different units. One of them is a regular terse research article, whereas the other one is the detailed and independently written explanations for the paper, so that readers of the short paper do not need to go over all the cumbersome computations. The goal is to clarify the dependence with respect to the time variable of some estimates about the convergence of the Trapezoidal Rule based Convolution Quadrature method applied to hyperbolic problems. This requires a careful investigation of the article of Lehel Banjai where the first convergence estimates were introduced, and of some technical results from a classical paper of Christian Lubich.\\
{\bf Key words.} Convolution quadrature, Laplace transforms, Trapezoidal rule \\
{\bf AMS Classification.}  26D10, 30A10, 44A10, 65L06
\end{abstract}

\begin{framed}
\begin{center}
The actual article
\end{center}
\end{framed}

\section{Introduction}

Without much practical motivation, let us explain what the goals of this article are. We first need to give a short introduction to the Convolution Quadrature (CQ) techniques to approximate causal distributional convolutions by the systematic use of the symbol (Laplace transform) of the operator. We will be interested in particular in the Trapezoidal Rule (TR) based CQ method. The reason for this is multiple. The original work of Lubich \cite{Lubich1994} extends his results \cite{Lubich1988} on multistep CQ for parabolic problems (parabolic character is reflected in having the Laplace transform of the operator extended to a sector around the negative axis) to hyperbolic problems (where the Laplace transform is defined only on a half plane). Because of Dahlquist barrier, only second order multistep CQ methods are available for hyperbolic problems, and the analysis in \cite{Lubich1994} excludes the TRCQ method for technical reasons. However, it is well known (and it has been tested repeatedly in the world of Time Domain Boundary Integral Equations -- TDBIE) that the TR based method outperforms the first order backward Euler method and BDF2 which is much too dispersive. Note that Runge-Kutta CQ schemes \cite{LuOs1993} with higher order and less dispersion are also available, and that a detailed time domain analysis is also missing from \cite{BaLu2011} and \cite{BaLuMe2011}.

As a warning to the reader, let us say that this paper is quite technical, but it closes an important question (left open in 
the monograph \cite{Sayas2016}) as how error estimates for TRCQ behave polynomially in time and there is no hidden Gronwall Lemma argument that would lead to exponential in time upper bounds. In the appendix of Banjai's paper \cite{Banjai2010}, which we are polishing up, the estimates are written for finite time intervals and the behavior with respect to the final time is not specified.

Let us now give a crash course on the mathematical aspects of TRCQ. For algorithmic and practical introductions to the CQ methods, we recommend \cite{BaSc2012} and \cite{HaSa2016}. For a detailed introduction to the distributional language required for a deep understanding of CQ applied to TDBIE, see \cite{Sayas2016}. Our starting point is a couple of Banach spaces $X$ and $Y$ and the space $\mathcal B(X,Y)$ of bounded linear operators $X\to Y$, whose norm will be denoted $\|\,\cdot\,\|_{X\to Y}$. The second ingredient is the symbol of a momentarily hidden convolutional operator: we assume that we have an analytic function
\begin{equation}\label{eq:pap:1.1}
\mathrm F:\C_+\to \mathcal B(X,Y), \qquad \C_+:=\{s\in \C\,:\, \Re\,s >0\}
\end{equation}
satisfying
\begin{equation}\label{eq:pap:1.2}
\| \mathrm F(s)\|_{X\to Y}\le C_F(\Re s)\,|s|^\mu \qquad \forall s\in \C_+
\end{equation}
where $C_F:(0,\infty)\to (0,\infty)$ is non-increasing and $C_F(x)\le c_0 \, x^{-m}$ for some $m\ge 0$ when $x$ is close to zero. We will be interested in symbols $\mathrm F$ where the parameter $\mu\ge 0$ in \eqref{eq:pap:1.2}, but we will show some results (based on \cite{Lubich1994}), where we will use negative values of $\mu$ as well. The TRCQ approximation of this symbol consists of defining
\begin{equation}\label{eq:pap:1.3}
\mathrm F_\kappa (s):=\mathrm F(s_\kappa), 
\qquad s_\kappa:=\frac{\delta(e^{-\kappa s})}\kappa=\frac2\kappa \tanh\left( \frac{\kappa s}2\right),
\qquad \delta(\zeta):=2\, \frac{1-\zeta}{1+\zeta}.
\end{equation}
Here $\kappa>0$ is the constant time-step (see more explanations later) of the underlying TR scheme (recall that $\delta$ in \eqref{eq:pap:1.3} is the characteristic function of the TR scheme). We will show that $s_\kappa\in \C_+$ for every $s\in \C_+$, so that the definition of $\mathrm F_\kappa$ makes sense, and we will also show that $\mathrm F_\kappa$ is a symbol with properties \eqref{eq:pap:1.1}-\eqref{eq:pap:1.2}, although with different parameters to those of $\mathrm F$.

Properties \eqref{eq:pap:1.1}-\eqref{eq:pap:1.2} ensure that $\mathrm F$ is the Laplace transform of a causal tempered $\mathcal B(X,Y)$-valued distribution, which we will name $f$. Moreover, $\mathrm F$ is the Laplace transform of the distributional time derivative of a certain order (depending on $\mu$) applied to a function $h:\R \to \mathcal B(X,Y)$ which is causal ($h\equiv 0$ in $(-\infty,0)$), continuous, and polynomialy bounded. See full details in \cite[Chapters 2 \& 3]{Sayas2016}. Under these conditions, we can define a convolutional product $f*g$ of the operator valued $f$ acting on a causal $X$-valued distribution $g$, outputting a causal $Y$-valued distribution. Informally, we are dealing with 
\[
(f*g)(t)=\int_0^t f(t-\tau) g(\tau)\mathrm d\tau.
\]
Similary, $\mathrm F_\kappa$ is the Laplace transform of a causal tempered $\mathcal B(X,Y)$-valued distribution $f_\kappa$, and the TRCQ approximation consists of substituting $f*g$ by $f_\kappa*g$. In practice, what is computed are the values
\begin{equation}
(f_\kappa*g)(t_n) \qquad t_n:=n\,\kappa, \qquad n\in \mathbb Z, \quad  n\ge 0,
\end{equation}
although the theory is developed for the full real distribution $f_\kappa*g$. The time-step values of $f_\kappa*g$ are given by the discrete convolution
\begin{equation}\label{eq:pap:1.5}
(f_\kappa*g)(t_n):=\sum_{m=0}^n \omega_{n-m}^F(\kappa) g(t_m),
\qquad
\mathrm F(\delta(\zeta)/\kappa)=\sum_{m=0}^\infty \omega_m^F(\kappa) \zeta^m.  
\end{equation}
In practice, the discrete convolutions \eqref{eq:pap:1.5} are computed using a parallel process, FFTs, and some kind of contour integration \cite{BaSa2008, BaSc2012,HaSa2016}. If applied to a linear system of ODEs with vanishing initial conditions, TRCQ is reduced to the TR scheme applied to the original system. One of the main field of applications of CQ for hyperbolic problems is in the area of TDBIE, using the language and ideas of the seminal papers of Bamberger and Ha-Duong \cite{BaDu1986a, BaDu1986b}. More examples, including coupled systems of wave equations in bounded domains with TDBIE in their exterior, can be found in \cite{LaSa2009}. More recently, a rich domain of applications of CQ has been opened in the numerical approximation of propagation of viscoelastic waves \cite{BrDuErSa2018}.

\section{The main theorem}

To state the main theorem, we will use the Sobolev-Bochner space
\[
W_+^n(\R;X):=\{g\in \mathcal C^{n-1}(\R;X) \,:\,g\equiv 0 \mbox{ in $(-\infty,0)$},\quad g^{(n)}\in L^1(\R;X)\}. 
\]
Note that if $g\in W_+^n(\R;X)$, then $g^{(m)}(0)=0$ for $m\le n-1.$ We will also need the $m$-th order linear differential operator
\[
(\mathcal P_m g)(t):=e^{-t}(e^{\,\cdot\,}g)^{(m)}(t)=\sum_{\ell=0}^m {m \choose \ell} g^{(\ell)}(t). 
\]
The remainder of this paper will consist of a proof of the next theorem. Some easy (but somewhat cumbersome details) are avoided. The reader is welcome to look for the arXiv version of this document to find a fastidiously detailed proof of each single step. 

\begin{theorem} \label{thm:pap:2.1}
Let $\mathrm F$ satisfy \eqref{eq:pap:1.1}-\eqref{eq:pap:1.2} for $\mu\ge 0$ and with $C_F$ fulfilling the conditions given after \eqref{eq:pap:1.2}. Let $f$ be the distributional inverse Laplace transform of $\mathrm F$ and $f_\kappa$ be its TRCQ approximation, i.e., the inverse Laplace transform of $\mathrm F_\kappa$ given in \eqref{eq:pap:1.3}, for any given time-step $\kappa\in (0,1]$. Consider the parameters
\begin{equation}\label{eq:pap:2.1}
m:=\lceil \mu \rceil, 
	\qquad \alpha:=\lfloor \mu-m\rfloor +5, 
	\qquad \beta:=\max\{ 2m+4,m+\alpha\}.
\end{equation}
For any  $g\in W_+^\beta(\R;X)$ and $t\ge 0$,we have
\begin{equation}\label{eq:pap:2.2}
\| (f_\kappa-f)*g(t)\|_Y
	\le C(t^{-1})\left(
		\int_0^t \| g^{(m+\alpha)}(\tau)\|_X\mathrm d\tau
			+\int_0^t \| \mathcal P_m g^{(m+4)}(\tau)\|_X \mathrm d\tau\right),
\end{equation}
where
\[
C(x):=C_F \left (\min\{x,1\}/4\right) \frac{C_\mu}{\min\{x^\varepsilon, 1\}},
	\qquad \varepsilon:=\max\{2m-\mu + 1, \floor{\mu}-\mu+3\}.
\]
\end{theorem}

Note that 
\[
\alpha=\begin{cases} 
	5, & \mu=m,\\
	4, & \mu\neq m,
		\end{cases}
	\qquad
\beta=\begin{cases}
	5, & \mu=0,\\
	2m+4, & \mu>0,
		\end{cases}
	\qquad
1 + \max\{m,1\} \le \varepsilon \le 2+\max\{m,1\}.
\]

\section{The TRCQ discrete derivative}

We now introduce some key functions for the estimates that follow. First of all, note that the function
\[
\omega\longmapsto 2\,\frac{\tanh(\omega/2)-\omega/2}{\omega^3}=
	\frac{\delta(e^{-\omega})-\omega}{\omega^3}=
	\sum_{\ell=0}^\infty b_\ell \,\omega^{2\ell}
\]
is even and analytic in $B(0;\pi)$. We then define
\[
D(\omega):=\sum_{\ell=0}^\infty \alpha_\ell\,\omega^{2\ell}, \qquad \alpha_\ell:=|b_\ell|,
\]
and note that $D$ is also analytic in the same disk and that the function $[0,\pi) \ni x\mapsto x^2 D(x)$ is strictly increasing, non-negative and diverges as $x\to \pi$. Therefore, there exists a unique 
\[
c_0\in (0,\pi), \qquad \mbox{such that} \qquad c_0^2 D(c_0)=1.
\]
Next, we define
\[
E_m(\omega):=\max\{ D^j(\omega) : j =1,\cdots,m \} \frac{(1+\omega^2)^m - 1}{\omega^2},
\]
and notice that $E_1 \equiv D$. Using these, we are going to present some properties of the characteristic function of the TR rule. At the end of this section we will give a technical result which will be a key tool in the proof of Theorem \ref{thm:pap:2.1}.

\begin{lemma} \label{lem:pap:3.1}
The following inequalities hold:
\begin{enumerate}
\item[{\rm (a)}]
$\Re \delta(e^{-z}) \geq \dfrac12 \min \{\Re z, 1\}$ for all $z \in \C_+$.
\item[{\rm (b)}]
$| \delta(e^{-z}) | \leq \dfrac8{ \min \{\Re z, 1\} }$ for all $z \in \C_+$.
\item[{\rm (c)}]
$|\delta^m(e^{-z}) - z^m| \leq E_m (|z|) |z|^{m+2}$ for all $m \geq 1$ and all $z \in \C_+$ with $|z| < \pi$.
\item[{\rm (d)}]
$\Re \dfrac{\delta(e^{-z})}{z} \geq 1 - |z|^2 D(|z|)$ for all $z \in \C_+$ with $0 < |z| < c_0 < \pi$.
\end{enumerate}\end{lemma}

\begin{proof}
For $|\zeta| < 1$, it is easy to verify that $\Re \frac{1-\zeta}{1+\zeta} \geq \frac{1 - |\zeta|}{1 + |\zeta|}.$ Using this with $\zeta = e^{-z}$ for $\Re z > 0$ and noting that $|e^{-z}| = e^{-\Re z}$, we write
\[
\frac12 \Re \delta(e^{-z}) \geq \frac{1 - e^{-\Re z}}{1 + e^{-\Re z}} = \tanh  \left( \frac{\Re z}2 \right).
\]
The proof of (a) follows then from
\begin{equation}\label{eq:3.100}
\tanh \frac x 2 \geq \frac14 \min\{ x,1 \}.
\end{equation}
We prove (b) by using \eqref{eq:3.100} together with the triangle and reverse triangle inequalities in the following way
\[
\frac12 | \delta(e^{-z}) | \le \frac{1 + e^{-\Re z}}{1 - e^{-\Re z}} = \coth  \left( \frac{\Re z}2 \right) \leq \frac4{\min\{\Re z, 1\}}.
\]

To show (c) and (d), we define 
\[
\C_+ \cap B(0;\pi)\ni z \longmapsto q (z) := \frac{\delta(e^{-z})-z}{z^3},
\]
and observe that $|q(z)| \leq D(|z|)$ which holds because of the definition of $D$. Using this, it is not hard to see that
\begin{align*}
	\left | \delta^m (e^{-z}) - z^m \right | 
		& =  \left | (z + z^3 q(z))^m - z^m \right | \\
		&\leq  \sum_{j=0}^{m-1} {m \choose j} |z|^j D^{m-j} (|z|) |z|^{3m -3j}
		\  \leq \  E_m(|z|) |z|^{m+2},
\end{align*}
which proves (c). For (d), we write
\[
\Re \frac{\delta(e^{-z})} z \geq 1 - |z|^2 |q(z)| \geq 1 - |z|^2 D(|z|) \qquad \forall z \in \C_+ \cap B(0;\pi).
\]
Note that the result is stated (and later used) only for $|z|<c_0<\pi$, which ensures that the right-hand side of the above inequality is positive. 
\end{proof}

The discrete version of this lemma will be a building block for the rest of this paper.

\begin{proposition}\label{prop:pap:3.2}
For $\kappa \in (0,1]$, the following inequalities hold:
\begin{enumerate}
\item[\rm (a)]
$\Re s_\kappa \geq \dfrac12 \min \{\Re s, 1\}$ for all $s \in \C_+$.
\item [\rm (b)]
$|s_\kappa| \leq \dfrac8 {\kappa^2 \min\{\Re s, 1\}}$ for all $s \in \C_+$.
\item [\rm (c)]
$|s_\kappa^m - s^m| \leq E_m (|\kappa s|) \kappa^2 |s|^{m+2}$ for all $m \geq 1$ and all $s \in \C_+$ with $|\kappa s| < \pi$.
\item [\rm (d)]
$\Re \dfrac{s_\kappa}{s} \geq 1 - |\kappa s|^2 D(|\kappa s|)$ for all $s \in \C_+$ with $0 < |\kappa s| < c_0 < \pi$.
\end{enumerate}
\end{proposition}

\begin{proof}
The result follows from Lemma \ref{lem:pap:3.1} by simply inserting $z=\kappa s$ and noting that $\kappa \in (0,1]$.
\end{proof}

\begin{lemma}\label{lem:pap:3.3}
For $g \in W^2_+(\R;X)$ and $\sigma > 0$ we have
\[
\int_{-\infty}^\infty \| \mr G(\sigma + i \omega) \|_X \mathrm d \omega \leq \frac\pi\sigma \int_0^\infty \| \ddot g (\tau) \|_X \mathrm d \tau.
\]
\end{lemma}

\begin{proof}
We can easily estimate
        	\begin{align*}
        		\int_{-\infty}^\infty   \| \mr G(\sigma + i \omega) \|_X \mathrm d \omega 
        		& = 
        		\int_{-\infty}^\infty  \frac1{\sigma^2 + \omega^2} \| (\sigma + i \omega)^2 \mr G(\sigma + i \omega) \|_X \mathrm d \omega \\
        		& \leq 
        		\sup_{\Re s = \sigma}  \| s^2 \mr G(s) \|_X \ \int_{-\infty}^\infty \frac{\mathrm d \omega}{\sigma^2 + \omega^2}
        		\ = \
        		\sup_{\Re s = \sigma} \| \mathcal L \{ \ddot g \}(s) \|_X \, \frac\pi\sigma.
        	\end{align*}
This finishes the proof.
\end{proof}

\begin{proposition}\label{prop:pap:3.4a}
If $g \in W_+^{2m+4}(\R;X)$ with $m \geq 0$, and 
\[
\mr G(s) := \mathcal L\{g\}(s) , \qquad \mr H(s):=(s_\kappa^m - s^m) \mr G(s),
\]
then for all $\sigma > 0$ we have
\[
\int_{-\infty}^\infty   \| \mr H(\sigma + i \omega)  \|_X \mathrm d \omega 
\leq 
\kappa^2 \frac {C_{m}^1} {\sigma \min\{\sigma^m,1\}} \int_0^\infty  \| (\mathcal P_m g^{(m+4)}) (\tau)  \|_X \mathrm d \tau,
\]
where $C_{m}^1$ is a positive constant depending only on $m$.
\end{proposition}
\begin{proof}
For the sake of convenience, we will abuse notation by eliminating the explicit dependence with respect to $\omega$ in 
\[
	s = s(\omega) := \sigma + i \omega, \qquad  
	s_\kappa = s_\kappa(\omega) := \frac1\kappa \delta(e^{-(\sigma + i \omega) \kappa}),
\]
where $\sigma>0$ is fixed.
We now take an arbitrary but fixed value of $c \in (0, \pi)$, and define the integration regions 
$ I^1 :=  \left\{ \omega \in \R\,: |\sigma  + i \omega| \leq c / \kappa \right\}$
and	
$I^2 :=  \left\{ \omega \in \R: |\sigma  + i \omega| \geq  c / \kappa \right\}$,
covering the entire real line. 
We split our target integral into three pieces, and work on them one by one
\begin{align*}
	\int_{-\infty}^\infty  \| \mr H(s) \|_X \dd \omega 
	\leq & 
	\int_{I^1} |s_\kappa^m - s^m| \, \| \mr G(s) \|_X \mathrm d \omega + \int_{I^2} |s_\kappa|^m \, \| \mr G(s) \|_X \mathrm d \omega + \int_{I^2} |s|^m \, \| \mr G(s) \|_X \mathrm d \omega.
\end{align*}
Since $|\kappa s| \leq c$ on $I^1$ and $E_m$ is increasing, Proposition \ref{prop:pap:3.2}(c)  yields
\[
\int_{I^1} |s_\kappa^m - s^m| \, \| \mr G(s) \|_X \mathrm d \omega 
 \leq 
\kappa^2 E_m(c)  \int_{-\infty}^\infty |s|^{m+2} \, \| \mr G(s) \|_X \mathrm d \omega.
\]
For the second integral, using Proposition \ref{prop:pap:3.2}(b) and the fact that $c \leq | \kappa s|$ on $I^2$, we have
\[
\int_{I^2} |s_\kappa |^m \, \| \mr G(s) \|_X \mathrm d \omega
			\leq \frac {\kappa^2}{c^{2m+2}}  \frac {8^m} {\min\{\sigma^m,1\}}  \int_{-\infty}^\infty \| s^{2m+2} \mr G(s) \|_X \mathrm d \omega.
\]
Lastly, the definition of $I^2$ implies that
\[
\int_{I^2} |s|^m \, \| \mr G(s) \|_X \mathrm d \omega 
		\leq  
		\frac{\kappa^2}{c^2}\int_{-\infty}^\infty \, \| s^{m+2} \mr G(s) \|_X \mathrm d \omega.
\]
Combining these three estimates we can write
\begin{equation} \label{eq:pap:3.4}
\int_{-\infty}^\infty \| \mr H(s) \|_X \mathrm d\omega 
\leq \kappa^2 \frac{e_{m,c}}{\min\{\sigma^m,1\}} \int_{-\infty}^\infty (1+|s|^m) \| s^{m+2} \mr G(s) \|_X \mathrm d \omega,
\end{equation}
where 
\[
e_{m,c} := \max \left\{ E_m(c) + \frac1{c^2},  \frac{8^m}{c^{2m+2}} \right\}.
\] 
The definition of $\mathcal P_m$ implies that 
\[
(1+s)^m s^{m+2} \mr G(s) = \mathcal L \{ \mathcal P_m g^{m+2} \} (s).
\]
Therefore, using that $1 + |s|^m \leq 2^{m/2} |1 + s|^m$ for $s \in \C_+$ and Lemma \ref{lem:pap:3.3}, we can write
\begin{align*}
	\int_{-\infty}^\infty  \| (1 + |s|)^m s^{m+2} \mr G(s) \|_X \mathrm d \omega
	& \leq
	2^{m/2} \int_{-\infty}^\infty  \| (1 + s)^m s^{m+2} \mr G(s) \|_X \mathrm d \omega \\
	& \leq
	2^{m/2} \frac \pi \sigma \int_0^\infty   \| (\mc P_m g^{(m+4)}) (\tau) \|_X \mathrm d \tau.
\end{align*}
This inequality and \eqref{eq:pap:3.4} prove the result with $C_m=2^{m/2} \pi \, e_{m,c}$. However, the dependence on $c\in (0,\pi)$ is limited to $e_{m,c}$, so we can eliminate $c$ by taking  $e_m := \ds \min_{c \, \in (0,\pi)} e_{m,c}$ in the bounds. 
\end{proof}

\section{Revisiting a result of Christian Lubich}
In this section, we work on some key results when $\mathrm F$ satisfies \eqref{eq:pap:1.1}-\eqref{eq:pap:1.2} with $\mu\le 0$. We start with showing that $\mr F_\kappa$, like $\mr F$, is the Laplace transform of a causal tempered $\mc B(X,Y)$-valued distribution.
In Proposition \ref{prop:pap:4.4} we revisit Lubich's  \cite[Theorem 3.1]{Lubich1994}, and prove it for $-1<\mu\le0$ by including the case $\mu = 0$ which was missing in that manuscript, and add the explicit dependence with respect to the time variable in the bounds. 

\begin{proposition} \label{prop:pap:4.1}
If $\mathrm F$ satisfies \eqref{eq:pap:1.1}-\eqref{eq:pap:1.2} with $\mu\le 0$, then
\begin{enumerate}
\item[{\rm (a)}]
$ \| \mr F_\kappa(s) \|_{X \to Y} \leq \Theta_1(\Re s) $ for all $s \in \C_+$.
\item[{\rm (b)}]
$ \| \mr F'(s) \|_{X \to Y}  \leq \Theta_2(\Re s) |s|^\mu$ for all $s \in \C_+$.
\item[{\rm (c)}] 
$ \| \mr F_\kappa (s) - \mr F(s)  \|_{X \to Y}  \leq \kappa^2 \Theta_2\left(\tfrac12\min\{\Re s,1\}\right) \Theta_3(|\kappa s|) |s|^{\mu +3}$ for all $s \in \C_+ \cap B(0;c_0/\kappa)$.
\end{enumerate}
In the above bounds
\begin{alignat*}{6}
\Theta_1(x) :=&  (\tfrac12 \min\{x, 1\})^\mu \, C_F(\tfrac12 \min\{x,1\}),\\
\Theta_2(x) :=& \frac{2^{1-\mu}}x  C_F(\tfrac12 x),\\
\Theta_3(x) :=& D(x) (1 - x^2 D(x))^\mu.
\end{alignat*}
The functions $\Theta_1$ and $\Theta_2$ are defined on $(0,\infty)$ and
they can be bounded by a negative power of $x$ as $x\to 0$. The function $\Theta_3$ 
is defined on $(0,c_0)$, is increasing, and when $\mu \neq 0$ it diverges as $x \to c_0$.
\end{proposition}

\begin{proof}
To prove (a), we first observe that Proposition \ref{prop:pap:3.2}(a) implies $s_\kappa \in \C_+$, therefore \eqref{eq:pap:1.2} gives
\[
\| \mr F_\kappa(s) \|_{X \to Y} \leq C_F(\Re s_\kappa) |s_\kappa|^\mu.
\]
The rest of the proof follows from Proposition \ref{prop:pap:3.2}(a) and the fact that $C_F$ and $(\cdot)^\mu$ are non-decreasing functions on $(0,\infty)$.

For (b), we use same ideas given in the proof of \cite[Proposition 4.5.3]{Sayas2016}. Defining the curve $\Xi(s) := \{ z \in \C : |z-s| = \frac12 \Re s \}$ with positive orientation, we write
\[
\mr F'(s) = \frac1{2\pi i} \int_{\Xi(s)} \frac{\mr F(z)}{(z-s)^2} \mathrm d z.
\]
We finish the proof of (b) using the fact that $\frac12 |s| \leq |z|$ and $\frac12 \Re s \leq \Re z$ for $z \in \Xi(s)$.

To show (c), we write the following by using the Mean Value Theorem
\[
\| \mr F_\kappa (s) - \mr F(s)  \|_{X \to Y}  \leq  \| \mr F'( \lambda s_\kappa + (1-\lambda) s) \|_{X \to Y} \, |s_\kappa - s|,
\]
for some $\lambda \in (0,1)$. Now, we define $z(s) := \lambda s_\kappa + (1-\lambda) s$, and use (a) to write
\[
\| \mr F'(z) \|_{X \to Y} \leq \Theta_2 (\Re z) |z|^\mu.
\]
This can be bounded by observing
\[
\Re z \geq \min\{\Re s_\kappa, \Re s\}, \qquad |z| \geq |s| \, \Re\frac z s \geq |s| \min\left\{\Re\frac{s_\kappa}s,1 \right\},
\]
and then using Proposition \ref{prop:pap:3.2}(a) and (d). We finish the proof by using Proposition \ref{prop:pap:3.2}(c) with $m=1$ and noting that $E_1 \equiv D$.

\end{proof}
\begin{lemma}\label{lem:pap:4.2}
The following holds for all $\sigma > 0$, $\alpha > 1$, $\kappa \in (0,1]$ and $c > 0$:
\begin{enumerate}
\item[{\rm (a)}] 
$\ds \int_{|\sigma + i \omega| \geq c/\kappa} |\sigma + i \omega|^{-\alpha} \, \mathrm d \omega 
\leq 
\frac{2 \alpha}{\alpha - 1}  \left( \frac{\kappa}{c} \right)^{\alpha-1}.$
\item[{\rm (b)}] 
$\ds \int_{-\infty}^\infty |\sigma + i \omega|^{-\alpha} \, \mathrm d \omega 
\leq 
\frac{2}{\sigma^\alpha} + \frac{2}{\alpha - 1}.$
\end{enumerate}
\end{lemma}
\begin{proof}
In order to prove (a), for fixed $\sigma, c$ and $\kappa$, we define the domains of integration
$I^1 :=  \left\{ \omega \in \R: |\sigma + i \omega| \geq  c/ \kappa, |\omega| \leq  c/ \kappa \right\}$
and 
$I^2 := \{ \omega \in \R: |\omega| \geq  c / \kappa \}$, 
which give
\[
\int_{|\sigma + i \omega| \geq c/\kappa} |\sigma + i \omega|^{-\alpha} \, \mathrm d \omega
		=
		\int_{I^1} |\sigma + i \omega|^{-\alpha} \, \mathrm d \omega
		+
		\int_{I^2} |\sigma + i \omega|^{-\alpha} \, \mathrm d \omega.
\]
We bound the first integral using the fact that $|\sigma + i \omega|^{-\alpha} \leq  (c/\kappa)^{-\alpha}$ on $I^1$. We rewrite the second integral using a change of variables and bound it in the following way
\[
\kappa^{\alpha-1} \int_{|\omega| \geq c} |\sigma + i \omega|^{-\alpha} \, \mathrm d \omega 
\leq 
\kappa^{\alpha-1} \int_{|\omega| \geq c} |\omega|^{-\alpha} \, \mathrm d \omega = \frac2{\alpha - 1} \left( \frac\kappa c \right)^{\alpha -1},
\]
which finishes the proof of (a). We prove (b) by simply writing
\[
 \int_{-\infty}^\infty |\sigma + i \omega|^{-\alpha} \, \mathrm d \omega 
		 \leq 
		 2 \int_0^1 \sigma^{-\alpha} \, \mathrm d \omega + 2 \int_1^\infty \omega^{-\alpha} \, \mathrm d \omega
		  =
		 \frac2 {\sigma^\alpha} + \frac2{\alpha-1}.
\]
\end{proof}

\begin{proposition}\label{prop:pap:4.3}
If $ \mr F$ satisfies \eqref{eq:pap:1.1}-\eqref{eq:pap:1.2} with $-1 < \mu \leq 0$ and $\alpha = \floor{\mu + 5}$, then for all $\sigma > 0$ we have
	\[
		\int_{-\infty}^\infty  \| \mathrm F_\kappa (\sigma + i \omega) - \mathrm F(\sigma + i \omega) \|_{X \to Y} \, |\sigma + i \omega|^{-\alpha} \, \mathrm d \omega 
		\leq 
		\kappa^2 \, C_\mu^1 \, C_2(\sigma),
	\]
where
\[
C_2(x) :=  \frac {C_F(\tfrac14 \min \{x,1\} )} {\min \{ x^{2 + \delta}, 1\} },  \qquad \delta := \floor\mu - \mu + 1 \in (0,1],
\] 
and $C_{\mu}^1$ is a positive constant depending only on $\mu$.
\end{proposition}
\begin{proof}
For fixed $\sigma, \kappa$ and $c \in (0,c_0)$ we will make use of the following domains of integration
\[
	I^1 :=  \left\{ \omega \in \R : | \sigma + i \omega | \leq c / \kappa \right\}, \qquad 
		I^2 :=  \left\{ \omega \in \R : | \sigma + i \omega | \geq  c / \kappa \right\},
\]
and the notation $s = s(\omega) := \sigma + i \omega$. Using these we can write
\begin{align*}
	\int_{-\infty}^\infty  \| \mr F_\kappa (s) - \mr F(s) \|_{X \to Y} \, |s|^{-\alpha} \, \dd \omega
		& \leq
		\int_{I^1}  \| \mr F_\kappa (s) - \mr F(s) \|_{X \to Y} \, |s|^{-\alpha} \, \dd \omega \nonumber \\
		& \quad + \int_{I^2}  \left(  \| \mr F_\kappa (s) \|_{X \to Y} +  \| \mr F(s) \|_{X \to Y}   \right) \, |s|^{-\alpha} \, \dd \omega.
\end{align*}
We bound the first integral on the right-hand side using Proposition \ref{prop:pap:4.1}(c) and Lemma \ref{lem:pap:4.2}(b)
\begin{align*}
	\int_{I^1}  \| \mr F_\kappa (s) - \mr F(s) \|_{X \to Y} \, |s|^{-\alpha} \, \dd \omega
	& \leq
	\kappa^2 \Theta_2 ( \min\{\sigma,1\}/2) \Theta_3 (c)  \left( \frac2{\sigma^{\alpha-\mu-3}} + \frac2{\alpha-\mu-4} \right)\\
	& \leq
	\kappa^2 C_F(\min\{\sigma, 1\}/4 ) \frac{e^1_\mu \, \Theta_3 (c)}{\min\{\sigma^{2+\delta},1\}},
\end{align*}
where $e^1_\mu$ is a positive constant depending only on $\mu$. Next, with the help of Proposition \ref{prop:pap:4.1}(a) and Lemma \ref{lem:pap:4.2}(a), the second integral is bounded in the following way
\begin{align*}
	\int_{I^2} \left( \| \mr F_\kappa (s) \|_{X \to Y} + \| \mr F(s) \|_{X \to Y} \right) \, |s|^{-\alpha} \, \dd \omega 
		& \leq
		\left( \Theta_1 (\sigma) + C_F(\sigma) \sigma^\mu \right)\frac{2 \alpha}{\alpha - 1}  \left( \frac{\kappa}{c} \right)^{\alpha-1} \nonumber \\
		& \leq
		\kappa^2 C_F(\min\{\sigma,1\}/2 ) \frac{e^2_\mu \, c^{1-\alpha}}{\min\{\sigma,1\} }.
\end{align*}
Here $e^2_\mu$ is a positive constant depending only on $\mu$. Combining these estimates we write
\[
\int_{-\infty}^\infty  \| \mr F_\kappa (s) - \mr F(s) \|_{X \to Y} \, |s|^{-\alpha} \, \dd \omega
		\leq
		\kappa^2 C_F(\min\{\sigma, 1\}/4 ) \frac{C_{\mu,c}}{\min\{\sigma^{2+\delta},1\}},
\]
with $C_{\mu,c} := e^1_\mu \, \Theta_3 (c) + e^2_\mu \, c^{1-\alpha}$. Note that this estimate holds for all $c \in (0,c_0)$, therefore we finish the proof by replacing the constant $C_{\mu,c}$ with $C_\mu^1 := \ds \min_{c \, \in (0,c_0)} C_{\mu,c}$.
\end{proof}

\begin{proposition}\label{prop:pap:4.4} If $\mr F$ satisfies \eqref{eq:pap:1.1}-\eqref{eq:pap:1.2} with $-1<\mu\le0$, and $g \in W_+^\alpha(\R;X)$ with $\alpha := \floor{\mu + 5}$, then
\[
\| (f_\kappa - f) * g(t) \|_Y \leq \kappa^2 \, C_\mu^2 \, C_2(t^{-1}) \int_0^t \| g^{(\alpha)} (\tau) \|_X \mathrm d \tau
\]
holds for all $t \geq 0$, where $C_\mu^2$ is a positive constant depending only on $\mu$.
\end{proposition}
\begin{proof}
For any $\sigma > 0$ and $t>0$, using the inverse Laplace transformation, we write
\[
\|  (f_\kappa - f) * g(t) \|_Y 
		\leq 
		\frac{e^{\sigma t}}{2\pi} \sup_{\Re s = \sigma} \| s^\alpha \mc L \{ g \}(s)  \|_X  \!\!
				\int_{-\infty}^\infty \| (\mr F_\kappa- \mr F)(\sigma + i \omega)\|_{X \to Y} \, | \sigma + i \omega |^{-\alpha} \, \dd \omega.
\]
Next, the definition of Laplace transformation together with Proposition \ref{prop:pap:4.3} and setting $\sigma  = t^{-1}$ give
\[
\|  (f_\kappa - f) * g(t) \|_Y 
		\leq 
		\kappa^2 \, C^2_\mu \, C_2(t^{-1}) \int_0^\infty \| g^{(\alpha)} (\tau) \|_X \mathrm d \tau,
\]
Now, we are going to obtain an integral bound over the interval $(0,t)$. For a fixed $t > 0$, we define the following function
	\begin{align*}
		p(\tau) := \left\{
		\begin{array}{ll}
		g(\tau), & \tau \leq t,\\[0.5em]
		\ds \sum_{\ell = 0}^{\alpha-1} \frac{(\tau - t)^\ell}{\ell !} g^{(\ell)}(t), & \tau \geq t.
		\end{array}\right.
	\end{align*}
It is not hard to see that $p \in W^{\alpha}_+(\R;X)$, in other words, $p$ satisfies the conditions of the proposition. Using the fact that $p \equiv g$ on $(-\infty,t)$ and $p^{(\alpha)} \equiv 0$ on $(t,\infty)$, we write
	\begin{align*}
		\|  (f_\kappa - f) * g(t) \|_Y = \|  (f_\kappa - f) * p(t) \|_Y
		& \leq 
		\kappa^2 \, C^2_\mu \, C_2(t^{-1}) \int_0^\infty \| p^{(\alpha)} (\tau) \|_X \, \dd \tau \nonumber \\
		& = 
		\kappa^2 \, C^2_\mu \, C_2(t^{-1})  \int_0^t \| g^{(\alpha)} (\tau) \|_X \, \dd \tau,
	\end{align*}
which finishes the proof.
\end{proof}

\section{Proof of Theorem \ref{thm:pap:2.1}}
In this section we prove the main theorem. We start with presenting a lemma to obtain upper-bounds integrated over the interval $(0,t)$ rather than $(0,\infty)$ in the proof of the main theorem.
\begin{lemma}\label{lem:pap:5.1}
Let $m\geq0$, $g \in W^{2m+4}_+(\R; X)$, $h := e^\punto g^{(m+4)}$, and $t > 0$ be a fixed real number. We define the function
	\[
		\R \ni \omega \longmapsto  j(\omega) := e^{-\omega} \ds \sum_{\ell = 0}^{m-1} \frac{(\omega - t)^\ell}{\ell !} h^{(\ell)}(t),
	\]
and, for $n \geq 1$, the integration operator
	\[
		(\partial^{-n} f)(\tau) := \int_t^\tau \int_t^{\omega_1} \cdots \int_t^{\omega_{n-1}} f(\omega_n) \, \mr d \omega_n \, \cdots \,\mr d \omega_2 \,\mr d \omega_1.
	\]
The function
	\[
		p(\tau) := \left\{
		\begin{array}{ll}
		g(\tau), 	& 	\tau \leq t,\\[0.5em]
		\ds \sum_{\ell = 0}^{m+3} \frac{(\tau - t)^\ell}{\ell !} g^{(\ell)}(t)
			 + (\partial^{-m-4} j) (\tau), 	& 	\tau \geq t,
		\end{array}\right.
	\]
satisfies that $p \in W^{2m+4}_+ (\R;X)$ and $\mc P_m p^{(m+4)} \equiv 0$ on $(t, \infty)$.
\end{lemma}

\begin{proof}
We observe that, for $0 \leq k \leq m+3$, the functions
\[
\R \ni \tau \longmapsto ( \partial^k \, \partial^{-m-4} \, j )(\tau) = (\partial^{-m-4 + k} \, j )(\tau)
\]
vanish when $\tau = t$. From there, it is not hard to see that $p \in \mc C^{m+3} (\R;X)$. Next, since $h \in \mc C^{m-1}(\R;X)$, we know that
\[
q(\tau) := \left\{
		\begin{array}{ll}
		h(\tau), 	& 	\tau \leq t,\\[0.5em]
		\ds \sum_{\ell = 0}^{m-1} \frac{(\tau - t)^\ell}{\ell !} h^{(\ell)}(t), 	& 	\tau \geq t,
		\end{array}\right.
\]
is also in $\mc C^{m-1}(\R;X)$, and so is $e^{- \punto} q = p^{(m+4)}$. This shows that $p \in \mc C^{2m+3} (\R;X)$. The function $j^{(m)} \in L^1(t,\infty;X)$ and therefore  $p^{(2m+4)} \in L^1 (\R;X)$.
The rest of the proof follows from the fact that
\[
( \mc P_m p^{(m+4)} )(\tau)  =
		e^{-\tau} \frac{\mr d^m}{\mr d\tau^m} \left(  \sum_{\ell = 0}^{m-1} \frac{(\tau - t)^\ell}{\ell !} h^{(\ell)}(t)  \right) (\tau) = 0 \qquad \forall \tau \in (t,\infty).
\]
\end{proof}

\begin{proposition}\label{prop:pap:3.4}
Let $g \in W_+^{2m+4}(\R;X)$ with $m \geq 0$, and $\mr F$ satisfy \eqref{eq:pap:1.1}-\eqref{eq:pap:1.2}. We define $(\partial_t^\kappa)^m g$ such that
\[
\mc L \{ (\partial_t^\kappa)^m g \}(s) := s_\kappa^m \mr G(s), 
	\qquad \mr G(s) := \mathcal L\{g\} (s).
\]
The following estimate holds for all $t > 0$
\[
 \| f * (  (\partial_t^\kappa)^m g - g^{(m)}  ) (t) \|_Y
 \leq 
 \kappa^2 C_1(t^{-1}) \sup_{\Re s = t^{-1}}  \| \mr F (s) \|_{X \to Y} \int_0^t  \| \mc P_m g^{(m+4)} (\tau) \|_X \mr d \tau,
\]
where
\[
C_1(x) := \frac{C_m}{x  \min \{x^m,1\} },
\]
and $C_m$ is a positive constant depending only on $m$.
\end{proposition}
\begin{proof}
For any $\sigma >0$, using the inverse Laplace transformation, we write
\[
 \| f * (  (\partial_t^\kappa)^m g - g^{(m)}  ) (t) \|_Y
 \leq 
\frac{e^{\sigma t}}{2 \pi} \sup_{\Re s = \sigma}  \| \mr F (s) \|_{X \to Y} \int_{-\infty}^\infty  \| \mr H (\sigma + i \omega) \|_X \mr d \omega,
\]
where $\mr H(s) := s_\kappa^m \mr G(s) - s^m \mr G(s)$. 
Here, with the help of Proposition \ref{prop:pap:3.4a} and inserting $\sigma = t^{-1}$, we obtain
\begin{equation}\label{eq:pap:3.5}
 \| f * (  (\partial_t^\kappa)^m g - g^{(m)}  ) (t) \|_Y
 \leq 
 \kappa^2 C_1(t^{-1}) \sup_{\Re s = t^{-1}}  \| \mr F (s) \|_{X \to Y} \int_0^\infty  \| \mc P_m g^{(m+4)} (\tau) \|_X \mr d \tau.
\end{equation}
Now, our goal is to have an integral bound over the interval $(0,t)$. To do that, for fixed $t>0$, we consider the function $p$ introduced in Lemma \ref{lem:pap:5.1}. Since $p \in W^{2m+4}_+(\R;X)$, in other words, it satisfies the conditions of this proposition, we can have the estimate \eqref{eq:pap:3.5} for $p$ as well. Therefore, using the properties of this function, we write
\begin{align*}
 \| f * (  (\partial_t^\kappa)^m g - g^{(m)}  ) (t) \|_Y 
 & =  
 \| f * (  (\partial_t^\kappa)^m p - p^{(m)}  ) (t) \|_Y\\
 & \leq
  \kappa^2 C_1(t^{-1}) \sup_{\Re s = t^{-1}}  \| \mr F (s) \|_{X \to Y} \int_0^\infty  \| \mc P_m p^{(m+4)} (\tau) \|_X \mr d \tau \\
 & =
  \kappa^2 C_1(t^{-1}) \sup_{\Re s = t^{-1}}  \| \mr F (s) \|_{X \to Y} \int_0^t  \| \mc P_m g^{(m+4)} (\tau) \|_X \mr d \tau.
\end{align*}
This finishes the proof.
\end{proof}

\begin{proof}{[Theorem 2.1]}
For $s \in \C_+$, we define
\[
\mr F^m (s) := s^{-m} \mr F(s) = \mc L \{ f^m \} (s), \qquad 
\mr F^m_\kappa (s) := s^{-m}_\kappa \mr F(s_\kappa) = \mc L \{ f^m_\kappa \} (s),
\]
and 
\[
\mr G(s) := \mc L\{ g \} (s), \qquad 
\mr H(s) := s^m_\kappa \mr G(s) - s^m \mr G(s) = \mc L \{ h \} (s).
\]
Using these definitions it is not hard to see that
\[
(f_\kappa  - f) * g = (f^m_\kappa - f^m) * g^{(m)} + f_\kappa^m * h.
\]
Now, we will obtain bounds for the terms on the right-hand side. For the first term, since  $g^{(m)} \in W^{\alpha}_+(\R;X)$ and 
\[
	\| \mr F^m \|_{X \to Y} \leq C_F(\Re s) |s|^{\mu - m} \qquad \forall s \in \C_+,
\]
where $\mu - m \in (-1,0]$, we can use Proposition \ref{prop:pap:4.4} to write
\begin{equation}\label{eq:pap:5.1}
	\| (f^m_\kappa - f^m) * g^{(m)} (t) \|_Y 
	\leq  
	 \kappa^2 C_{\mu-m}^2 C_2(t^{-1}) \int_0^t \| g^{(m+\alpha)} (\tau) \|_X \, \dd \tau.
\end{equation}
Next, we bound the second term using Proposition \ref{prop:pap:3.4} in the following way
\[
\| (f_\kappa^m * h)(t) \|_Y 
		\leq 
		\kappa^2 C_1(t^{-1}) \sup_{\Re s = t^{-1}}  \| \mr F_\kappa^m (s) \|_{X \to Y} \,  \int_0^t \| \mc P_m^{(m+4)} (\tau) \|_X  \, \dd \tau.
\]
Here, using the definition of $C_1$ and Proposition \ref{prop:pap:4.1}(a) we have
\[
C_1(x) \sup_{\Re s = x}  \| \mr F_\kappa^m (s) \| \leq C_F(\min\{x,1\}/2) \frac{C_m \, 2^{-\mu +m}}{x \min\{ x^{2m-\mu},1\}},
\]
for all $x>0$. Combining this with \eqref{eq:pap:5.1} finishes the proof.
\end{proof}

\bibliographystyle{abbrv}
\bibliography{TRCQ}

\begin{framed}
\begin{center}
A longer script with all the details and no motivation. This part of the document starts from scratch and goes slowly over all the details of the shorter paper. The order of the arguments is different but notation is the same. The list of references is the same one as in the first part of the paper. 
\end{center}
\end{framed}

\setcounter{section}{0}
\section{General concepts and definitions} \label{sec:1}
We start with defining the set
	\begin{align*} \label{eq:1.1}
		\C_+ := \{ s \in \C : \Re s > 0 \}.
	\end{align*}
Following the notation of \cite{Sayas2016}, we assume $\mr F : \C_+ \to \mc B (X,Y)$ is analytic and satisfies
	\begin{align}
		 \| \mr F(s) \| \leq C_F(\Re s) |s|^\mu \qquad \forall s \in \C_+,
	\end{align}
where
\begin{itemize}
\item
$ \| \mr F(s) \| :=  \| \mr F(s) \|_{X \to Y}$,
\item
$C_F$ is non-increasing on $(0,\infty)$,
\item
$C_F(\sigma^{-1})$ is polynomially bounded at $\sigma = 0$.
\end{itemize}
The transfer function for the Trapezoidal rule is 
	\begin{align*}
		\delta(\zeta) := 2\frac{1-\zeta}{1+\zeta}.
	\end{align*}
Equivalently, we can also write
	\begin{align}
		\delta(e^{-z}) = 2 \tanh  \frac z 2 . \label{eq:1.2b}
	\end{align}
We will use the following notation in our temporal approximations 
	\begin{align}
		s_\kappa &:= \frac1\kappa \delta(e^{-\kappa s}) = \frac2\kappa \tanh \left(\frac{\kappa s}{2} \right), \label{eq:1.3b} \\
		\mr F_\kappa (s) &:= \mr F(s_\kappa), \nonumber
	\end{align}
and assume that temporal discretization parameter is bounded
	\begin{align*}
		0 \leq \kappa \leq 1.
	\end{align*}

\section{Technical results}
\begin{lemma} \label{lem:2.1}
The following properties hold for all $x \geq 0$:
\begin{enumerate}
\item[\rm (a)]
$\tanh x \geq \dfrac12  \min \{ 1, x \}$.
\item[\rm (b)]
$\coth x  \leq \dfrac2{\min\{ 1, x \}}$.
\item[\rm (c)]
$\tanh \dfrac{x}2 \geq \dfrac14 \min \{ 1, x\}$.
\item[\rm (d)]
$\coth \dfrac{x}2  \leq \dfrac4{\min\{1,x\}}$.
\end{enumerate}
\end{lemma}

\begin{proof}
To prove (a), we first observe that $\tanh$ is increasing (see Figure \ref{fig:1}) and 
	\begin{align}
		\tanh 1 \approx 0.76 > 0.5,	\label{eq:2.2}
	\end{align} 	
therefore
	\begin{align*}
		\tanh x \geq \tanh 1 > \frac12 = \frac12 \min \{1,x\} \qquad \forall x \geq 1.
	\end{align*}
\begin{figure}[htbp]
\begin{center}
\includegraphics[clip,scale=0.4]{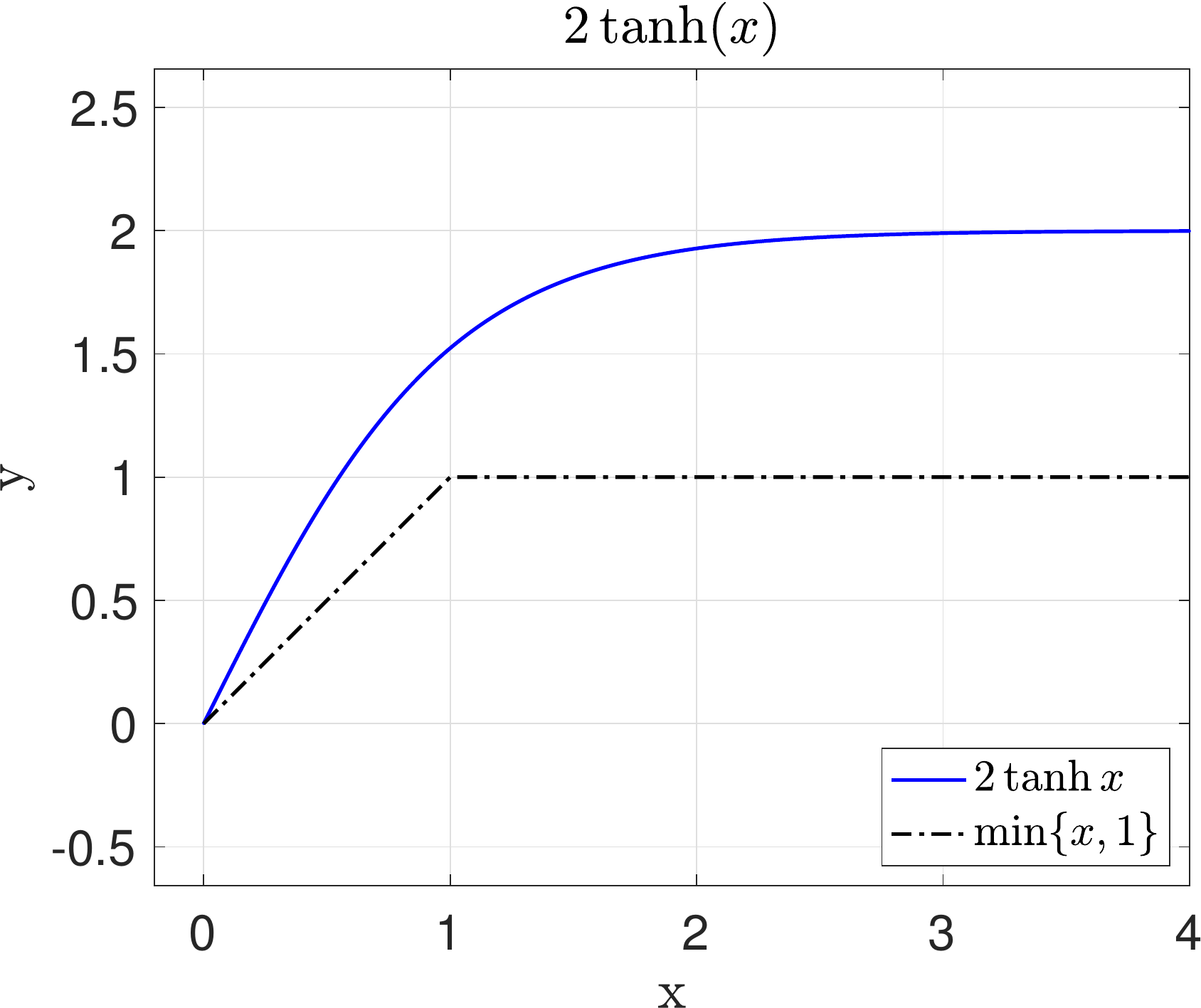}
\end{center}
\caption{$2\tanh(x)$ and $\min\{1,x\}$} \label{fig:1}
\end{figure}

For the case $x \leq 1$, we define the function
	\begin{align*}
		f (x) := \tanh x - \frac x 2.
	\end{align*}
Since 	
	\begin{align*}
		f'' (x) =  - 2 \frac  {\sinh x }{\cosh^3 x} \leq 0 \qquad \forall x \geq 0,
	\end{align*}
we know $f$ is concave. Using \eqref{eq:2.2}, we have
	\begin{align*}
		f(0) = 0, \quad f(1) = \tanh 1 - \frac12 > 0 ,
	\end{align*}
and, therefore,
	\begin{align*}
		f (x) = \tanh x - \frac x 2 > 0 \qquad \forall x \in (0,1),
	\end{align*}
which finishes (a). Part (b) follows from (a). For (c), using (a), we write
	\begin{align*}
		\tanh \frac x 2 \geq \frac12 \min \{1, \frac x 2 \} \geq \frac12 \min \{ \frac12, \frac x 2 \} = \frac14 \min \{1, x\}.
	\end{align*}
This also implies (d).
\end{proof}

\begin{proposition} \label{prop:2.2}
The following holds:
\begin{enumerate}
\item[\rm (a)]
$\Re \delta(e^{-z}) \geq \dfrac12 \min \{1, \Re z\}$ for all $z \in \C_+$.
\item[\rm (b)]
$| \delta(e^{-z}) | \leq \dfrac8{ \min \{1, \Re z\} }$ for all $z \in \C_+$.
\item[\rm (c)]
$|\delta(e^{-z}) - z| \leq D (|z|) |z|^3$ for all $z \in \C_+$ with $|z| < \pi$.
\item[\rm (d)]
$|\delta^m(e^{-z}) - z^m| \leq E_m (|z|) |z|^{m+2}$ for all $m \geq 1$ and all $z \in \C_+$ with $|z| < \pi$.
\item[\rm (e)]
$\Re \dfrac{\delta(e^{-z})}{z} \geq 1 - |z|^2 D(|z|)$ for all $z \in \C_+$ with $0 < |z| < c_0 < \pi$.
\end{enumerate}
Here 
$c_0 \in (0,\pi)$ is such that $c^2_0 D(c_0) = 1$, and
	\begin{align} \label{eq:2.8}
		D(\sigma) := \sum_{\ell = 0}^\infty \alpha_\ell \, \sigma^{2 \ell}, \qquad 
		E_m(\sigma) := \max\{D^j(\sigma) : j =1,\ldots,m \}\frac{(1 + \sigma^2)^m - 1}{\sigma^2},
	\end{align}
where $D$ is analytic in $B(0;\pi)$ with $\alpha_\ell \geq 0$ for all $\ell$. The functions
\[
(0,\pi) \ni \sigma \longmapsto D(\sigma), \qquad (0,\pi) \ni \sigma \longmapsto E_m(\sigma)
\]
are increasing and they diverge as $\sigma \to \pi^-$. Also note that $E_1 \equiv D$.
\end{proposition}

\begin{proof}
To prove (a), we first observe that the following holds for all $|\zeta| < 1$
	\begin{align} \label{eq:2.9}
		\Re \frac{1-\zeta}{1+\zeta} = \frac{1 - |\zeta|^2}{1 + 2 \Re \zeta + |\zeta|^2} \geq \frac{1 - |\zeta|^2}{(1 + |\zeta|)^2} = \frac{1 - |\zeta|}{1 + |\zeta|}.
	\end{align}
Now, when $\Re z > 0$ we have $| e^{-z} | < 1$. Therefore we can use \eqref{eq:2.9} with $\zeta = e^{-z}$, which gives
	\begin{align*}
		\frac12 \Re \delta (e^{-z}) \geq \frac{1 - |e^{-z}|}{1 + |e^{-z}|} = \frac{1 - e^{-\Re z}}{1 + e^{-\Re z}} = \tanh  \left( \frac {\Re z}2 \right).
	\end{align*}
The rest of the proof of (a) follows from Lemma \ref{lem:2.1}(c).

To prove (b), using the triangle and reverse triangle inequalities, we write
	\begin{align*}
		\frac{|1 - \zeta|}{|1 + \zeta|} \leq \frac{1 + |\zeta|}{1 - |\zeta|} \qquad |\zeta| < 1,
	\end{align*}
and therefore
	\begin{align*}
		| \delta(e^{-z}) |  = 2 \frac{| 1 - e^{-z}|}{|1 + e^{-z}|} 
		\leq 2 \frac{ 1 +  |e^{-z}|}{1 - |e^{-z}|} 
		\leq 2 \frac{ 1 + e^{-\Re z}}{1 - e^{-\Re z}} = 2 \coth \left( \frac{\Re z} 2 \right).
	\end{align*}
We then apply Lemma \ref{lem:2.1}(d) to finish the proof.

To show (c), we write the following (see \eqref{eq:1.2b})
	\begin{align*}
		\delta(e^{-z}) - z = 2 \left( \tanh \frac z 2 - \frac z 2 \right),
	\end{align*}
and define $\omega := \dfrac z 2$. We know $\tanh$ is an analytic odd function in $B(0;\frac\pi2)$, and $ \tanh'(0) = 1$, hence
	\begin{align} \label{eq:2.14}
		f(\omega) := \frac{\tanh \omega - \omega}{\omega^3} = \sum_{\ell = 0}^\infty a_\ell \, \omega^{2 \ell}
	\end{align}
is an analytic even function in the same domain. As a consequence, we know that
	\begin{align*}
		\frac14 f \left( \frac z 2 \right) = \sum_{\ell = 0}^\infty \frac{a_\ell}{2^{2 \ell + 2}} \, z^{2 \ell}
	\end{align*}
is analytic in $B(0;\pi)$. Using this, we can now write
	\begin{align} \label{eq:2.17}
		 \left| \frac{\delta(e^{-z}) - z}{z^3}   \right| = \frac14  \left| \frac{\tanh \frac z 2 - \frac z 2}{ \left( \frac z 2 \right)^3 }  \right| 
		 = \frac14 \left| f \left( \frac z 2 \right) \right|
		 \leq \sum_{\ell = 0}^\infty  \left| \frac{a_\ell}{2^{2 \ell + 2}} \right| |z|^{2 \ell}.
	\end{align}
To finish this part of the proof, we define $\alpha_\ell =  \left| \dfrac{a_\ell}{2^{2 \ell + 2}} \right| $, which gives
	\begin{align} \label{eq:2.18}
		D(\sigma) = \sum_{\ell = 0}^\infty \alpha_\ell \, \sigma^{2 \ell} = \sum_{\ell = 0}^\infty  \left| \dfrac{a_\ell}{2^{2 \ell + 2}} \right| \, \sigma^{2 \ell}.
	\end{align}

For (d), we first write 
	\begin{align*}
		\delta (e^{-z}) = z + q(z),
	\end{align*}
where, because of (c), we know that
	\begin{align} \label{eq:2.20}
		|q(z)| \leq D(|z|) |z|^3 \qquad \forall |z| \leq \pi.
	\end{align}
Next, we compute
	\begin{align*}
		\delta^m (e^{-z}) = z^m + \sum_{j=0}^{m-1} {m \choose j} z^j q^{m-j} (z),
	\end{align*}
and, using \eqref{eq:2.20}, we obtain
	\begin{align*}
		\left | \delta^m (e^{-z}) - z^m \right | 
			&\leq  \sum_{j=0}^{m-1} {m \choose j} |z|^j D^{m-j} (|z|) |z|^{3m -3j} \nonumber \\
			& \leq  \max\{D^j(|z|)  : j =1,\ldots,m \} \ |z|^m  \sum_{j=0}^{m-1} {m \choose j} |z^2|^{m-j},
	\end{align*}
which holds for all $|z| < \pi$. Finally, using the fact that
	\begin{align*}
		\frac1{|z|^2}\sum_{j=0}^{m-1} {m \choose j} |z^2|^{m-j} = \frac{(1 + |z|^2)^m - 1}{|z|^2},
	\end{align*}
and the definition of the function $E_m$, we write
	\begin{align*}
		|\delta^m(z) - z^m| \leq E_m (|z|) |z|^{m+2}.
	\end{align*}

To prove (e), using \eqref{eq:1.2b} and \eqref{eq:2.14}, we write
	\begin{align*}
		\frac{\delta(e^{-z})}{z} = \frac{2 \tanh (z/2)}{z} \quad \forall z \in \C_+, 
		\qquad \frac{\tanh \omega}{\omega} = 1 + \omega^2 f(\omega) \quad \forall \omega \in B(0;\tfrac\pi2).
	\end{align*}
Therefore, for all $z \in B(0;\pi)$, we have
	\begin{align} \label{eq:2.27}
		\Re \frac{\delta(e^{-z})} z    = \Re  \left( 1 + z^2 \frac14 f \left( \frac z 2 \right ) \right) 
							 \geq 1 - |z|^2 \frac14  \left| f \left( \frac z 2 \right) \right|  
							\geq 1 - |z|^2 D(|z|),
	\end{align}
where we used \eqref{eq:2.17} and \eqref{eq:2.18} for the last inequality. Although \eqref{eq:2.27} holds for all $z \in B(0;\pi)$, it is not useful to have a negative lower bound. To avoid that, we investigate the function 
	\begin{align*}
		[0,\pi) \ni \sigma \mapsto \sigma^2 D(\sigma).
	\end{align*}
This function is increasing, since $D$ is increasing (see its definition \eqref{eq:2.8}), it vanishes at $\sigma=0$, and diverges as $\sigma \to \pi^-$. Hence, there is a unique $c_0 \in (0,\pi)$ such that 
	\begin{align*}
		c_0^2 D(c_0) = 1,
	\end{align*}
and we will consider the inequality \eqref{eq:2.27} when $0<|z|< c_0<\pi$.
\end{proof}

Following result is the discrete counterpart of Proposition \ref{prop:2.2}.
\begin{proposition} \label{prop:2.3}
For $s_\kappa$ as in \eqref{eq:1.3b}, we have	
\begin{enumerate}
\item[\rm (a)]
$\Re s_\kappa \geq \dfrac12 \min \{1, \Re s\}$ for all $s \in \C_+$.
\item [\rm (b)]
$|s_\kappa| \leq \dfrac8 {\kappa^2 \min\{1, \Re s \}}$ for all $s \in \C_+$.
\item [\rm (c)]
$|s_\kappa^m - s^m| \leq E_m (|\kappa s|) \kappa^2 |s|^{m+2}$ for all $m \geq 1$ and all $s \in \C_+$ with $|\kappa s| < \pi$.
\item [\rm (d)]
$\Re \dfrac{s_\kappa}{s} \geq 1 - |\kappa s|^2 D(|\kappa s|)$ for all $s \in \C_+$ with $0 < |\kappa s| < c_0 < \pi$.
\end{enumerate}
\end{proposition}
\begin{proof}
Part (a) follows from Proposition \ref{prop:2.2}(a) and the fact that $\kappa \leq 1$
	\begin{align*}
		\Re s_\kappa = \frac1\kappa \Re \delta (e^{-\kappa s}) 
			& \geq \frac1{2\kappa} \min \{1, \kappa \Re s\} = \frac12 \min \{\frac1\kappa, \Re s\} \nonumber \\
			& \geq \frac12 \min \{1, \Re s\}.
	\end{align*}
Similarly, for (b), we use Proposition \ref{prop:2.2}(b) to write
	\begin{align*}
		| s_\kappa | = \frac1\kappa \left| \delta (e^{-\kappa s}) \right | 
			& \leq \frac{1}{\kappa} \frac{8}{\min \{ 1, \kappa \Re s \}} \nonumber \\
			& \leq \frac{1}{\kappa^2} \frac{8}{\min \{ \frac1\kappa, \Re s \}} \leq \frac{1}{\kappa^2} \frac{8}{\min \{ 1, \Re s \}}.
	\end{align*}
To prove (c), we use Proposition \ref{prop:2.2}(d) to show
	\begin{align*}
		| s_\kappa^m - s^m |  = \frac1{\kappa^m} | \delta^m(e^{-\kappa s}) - (\kappa s)^m | 
						 \leq \frac1{\kappa^m} E_m(|\kappa s|) |\kappa s|^{m+2}.
	\end{align*}
Finally, to prove (d), with the help of Proposition \ref{prop:2.2}(e) we write
	\begin{align*}
		\Re \dfrac{s_\kappa}{s}  = \Re \dfrac{\delta(e^{-\kappa s})}{\kappa s} \geq 1 - |\kappa s|^2 D(|\kappa s|),
	\end{align*}
which finishes the proof.
\end{proof}

\section{Discrete differentiation}
We will call a function $g : \R \to X$ \textit{causal}, if $g(t) = 0$ for all $t < 0$. Next, we define the following Sobolev spaces for $m \geq 1$ and Hilbert space $X$
	\begin{align*}
		W^{m}_+(\R; X) := \{ g \in \mc C^{m-1}(\R;X) :  g \text{ causal}, \quad g^{(m)} \in L^1(\R; X)  \}.
	\end{align*}
Note that, for $g \in W^{m}_+(\R; X)$, we have
	\begin{align*}
		g(0) = \ldots = g^{m-1}(0) = 0, \qquad g \in \mc C^{m-1}(\R; X),
	\end{align*}
and we know that the space
	\begin{align*}
		\mc C^m_+ (\R; X) := \{ g \in \mc C^m(\R; X): g \text{ causal} \}
	\end{align*}
is dense in $W^m_+(\R; X)$. We will also make use of the following definition
	\begin{align*}
		\ul \sigma := \min\{1, \sigma\}.
	\end{align*}

Before moving to the main result of this section, we are going to present some technical tools. The following lemma establishes a way of going from an integral in the Laplace domain to an integral in the time domain with a cost of two derivatives.
\begin{lemma} \label{lem:3.1}
For $g \in W^{2}_+(\R;X)$ and $\sigma > 0$ we have
	\begin{align*}
		\int_{-\infty}^\infty   \| \mr G(\sigma + i \omega) \|_X \dd \omega 
		\leq 
		\frac{\pi}{\sigma} \int_0 ^\infty  \| \ddot g(\tau) \|_X \dd \tau.
	\end{align*}
\end{lemma}
\begin{proof}
We start with writing
	\begin{align*}
		\int_{-\infty}^\infty   \| \mr G(\sigma + i \omega) \|_X \dd \omega 
		& = 
		\int_{-\infty}^\infty  \frac{1}{\sigma^2 + \omega^2} \| (\sigma + i \omega)^2 \mr G(\sigma + i \omega) \|_X \dd \omega \nonumber \\
		& \leq 
		\sup_{\Re s = \sigma}  \| s^2 \mr G(s) \|_X \ \int_{-\infty}^\infty \frac{\dd \omega}{\sigma^2 + \omega^2}.
	\end{align*}
We can evaluate the integral on the right-hand side exactly
	\begin{align*}
		2 \int_{0}^\infty \frac{\dd \omega}{\sigma^2 + \omega^2} 
		= 
		\frac 2 \sigma \lim_{\omega \to \infty}  \arctan \frac \omega \sigma
		= \frac \pi \sigma.
	\end{align*}
For the other part, using the definition of the Laplace transform we write
	\begin{align*}
		\sup_{\Re s = \sigma}  \| s^2 \mr G(s) \|_X 
		= 
		\sup_{\Re s = \sigma}  \| \mc L \{ \ddot g \}(s) \|_X
		 \leq 
		 \int_0 ^\infty  \| \ddot g(\tau) \|_X \dd \tau.
	\end{align*}
This finishes the proof.
\end{proof}

\begin{lemma} \label{lem:3.2}
For all $z \in \C_+$ and $m \geq 1$ we have
	\begin{align} \label{eq:3.10}
		1 + |z|^m \leq 2^{m/2} |1 + z|^m.
	\end{align}
\end{lemma}
\begin{proof}
Using the Cauchy-Schwarz inequality and the fact that $\Re z > 0$, we write
	\begin{align*}
		(1 + |z|)^2 = 1 + 2 |z| + |z|^2 	& \leq 2 (1 + |z|^2) \nonumber \\
								& < 2 (1 + 2 \Re z + |z|^2)  = 2 |1 + z|^2,
	\end{align*}
which proves \eqref{eq:3.10} for $m=1$. When $m > 1$, we do the following
	\begin{align*}
		1 + |z|^m \leq (1 + |z|)^m \leq  \left( \sqrt 2 |1 + z| \right)^m,
	\end{align*}
where for the last inequality we use the case $m=1$.
\end{proof}
Next result is a key ingredient in the proof of Theorem \ref{thm:5.3}. We introduce the following notation
	\begin{align} \label{eq:3.13}
		(\mc P_m g)(t) &:= e^{-t}(e^\punto g)^{(m)}(t) = \sum_{\ell = 0}^m {m \choose \ell} g^{(\ell)}(t).
	\end{align}
\begin{proposition}\label{prop:3.3}
Let $m\geq0$, $g \in W_+^{2m+4}(\R; X)$, and
	\begin{align*}
		\mr G := \mc L \{ g \}, \qquad \mr H(s) := (s^m_\kappa - s^m) \mr G(s).
	\end{align*}
For all $\sigma > 0$, we have
	\begin{align*}
		\int_{-\infty}^\infty   \| \mr H(\sigma + i \omega)  \|_X \dd \omega \leq \kappa^2 \frac {C_m^1} {\sigma \underline{\sigma}^m} \int_0^\infty  \| (\mc P_m g^{(m+4)}) (\tau)  \|_X \dd \tau,
	\end{align*}
where
	\begin{align} \label{eq:3.18}
		C_m^1 := \pi \, 2^{m/2} \min_{c \in (0,\pi)} \max\left\{ E_m(c) + \frac1{c^2}, \frac{8^m}{c^{2m+2}} \right\}.
	\end{align}
\end{proposition}
\begin{proof}
We start with fixing $\sigma > 0$. For an arbitrary and fixed $c \in (0, \pi)$, we define regions of the integration
	\begin{align*}
		I^1 :=  \left\{ \omega \in \R : |\sigma  + i \omega| \leq \frac c \kappa \right\}, \qquad
		I^2 :=  \left\{ \omega \in \R: |\sigma  + i \omega| \geq \frac c \kappa \right\}.
	\end{align*}
We also introduce the following notation for convenience
	\begin{align*}
		s = s(\omega) := \sigma + i \omega, \qquad  s_\kappa = s_\kappa(\omega) := \frac1\kappa \delta(e^{-(\sigma + i \omega) \kappa}).
	\end{align*}
Next, we split our target integral into three pieces
	\begin{align} \label{eq:3.21}
		\int_{-\infty}^\infty  \| \mr H(\sigma + i \omega) \|_X \dd \omega 
		\leq & 
		\int_{I^1} |s_\kappa^m - s^m| \, \| \mr G(\sigma + i \omega) \|_X \dd \omega \nonumber \\
		& + \int_{I^2} |s_\kappa^m | \, \| \mr G(\sigma + i \omega) \|_X \dd \omega + \int_{I^2} |s^m| \, \| \mr G(\sigma + i \omega) \|_X \dd \omega \nonumber \\
		 = & A_1 + A_2 + A_3.
	\end{align}
Now we are going to work on integrals $A_1,A_2,A_3$ separately and establish some bounds in the Laplace domain.

When $\omega \in I^1$, we have $|\kappa s| \leq c$. Therefore, using Proposition \ref{prop:2.3}(c) we have
	\begin{align*}
		|s_\kappa^m - s^m| \leq E_m(|\kappa s|) \kappa^2 |s|^{m+2} \leq E_m(c) \kappa^2 |s|^{m+2},
	\end{align*}
where we used the fact that $E_m$ is an increasing function (see \eqref{eq:2.8}). Hence, we can write
	\begin{align} \label{eq:3.22}
		A_1 
		= 
		\int_{I^1} |s_\kappa^m - s^m| \, \| \mr G(\sigma + i \omega) \|_X \dd \omega 
		\leq 
		\kappa^2 E_m(c)  \int_{-\infty}^\infty s^{m+2} \, \| \mr G(\sigma + i \omega) \|_X \dd \omega.
	\end{align}
Next, we first note that Proposition \ref{prop:2.3}(b) and the definition of $I^2$ imply
	\begin{align} \label{eq:3.23}
		\qquad |s_\kappa| \leq \frac 8 {\kappa^2 \underline\sigma }, \qquad \frac{1}{\kappa} \leq \frac{|s|} c \qquad \forall\omega \in I^2.
	\end{align}
Therefore, we have
	\begin{align*}
		|s_\kappa^m| 
		\leq  
		\left( \frac 8 {\kappa^2 \ul \sigma} \right)^m 
		= 
		\kappa^2  \left( \frac8{\ul \sigma} \right)^m \frac1{\kappa^{2m+2}}
		\leq
		\kappa^2 \left( \frac8{\ul \sigma} \right)^m \frac{|s|^{2m+2}}{c^{2m+2}},
	\end{align*}
which then gives the following bound for $A_2$
	\begin{align} \label{eq:3.24}
		A_2 	= \int_{I^2} |s_\kappa^m | \, \| \mr G(\sigma + i \omega) \|_X \dd \omega
			\leq \frac {\kappa^2}{c^{2m+2}} \left( \frac 8 {\underline\sigma} \right)^m  \int_{-\infty}^\infty \| s^{2m+2} \mr G(\sigma + i \omega) \|_X \dd \omega.
	\end{align}
For the last integral, the fact that
	\begin{align*}
		1 \leq \frac {|\kappa s|} c \qquad \forall \omega \in I^2
	\end{align*}
simply implies
	\begin{align} \label{eq:3.25}
		A_3 = \int_{I^2} |s^m| \, \| \mr G(\sigma + i \omega) \|_X \dd \omega 
		\leq  
		\frac{\kappa^2}{c^2}\int_{-\infty}^\infty \, \| s^{m+2} \mr G(\sigma + i \omega) \|_X \dd \omega.
	\end{align}
Now, we combine the bounds given in \eqref{eq:3.22}, \eqref{eq:3.24} and \eqref{eq:3.25} to write
	\begin{align} \label{eq:3.26}
		A_1 + A_2 + A_3 \leq \kappa^2 \frac{e_{m,c}} {\ul \sigma^m} \int_{-\infty}^\infty (1 + |s|^m)  \| s^{m+2} \mr G(s) \|_X \dd \omega,
	\end{align}
where
	\begin{align*}
		e_{m,c} := \max \left\{ E_m(c) + \frac1{c^2}, \ \frac{8^m}{ c^{2m+2}} \right\}.
	\end{align*}
Next, we are going to obtain a time domain bound for \eqref{eq:3.26}. To do that, we first observe that Lemma \ref{lem:3.2} gives
	\begin{align} \label{eq:3.28}
		(1 + |s|^m) & \| s^{m+2} \mr G(s) \|_X 
		\leq
		 2^{m/2} \| (1 + s)^m s^{m+2} \mr G(s) \|_X,
	\end{align}
and the definition of $\mc P_m$ in \eqref{eq:3.13} implies (note the extra $s^2$)
	\begin{align*}
		\mc L \{ \mc P_m g^{(m+4)} \}(s) = (1 + s)^m \mc L \{ g^{(m+4)} \}(s) = (1 + s)^m s^{m+4} \mr G(s).
	\end{align*}
Therefore, with the help of Lemma \ref{lem:3.1} we can write
	\begin{align} \label{eq:3.30}
		\int_{-\infty}^\infty   \| (1 + s)^m & s^{m+2} \mr G(s) \|_X \dd \omega
		\leq 
		\frac \pi \sigma \int_0^\infty   \| (\mc P_m g^{(m+4)}) (\tau) \|_X \dd \tau,
	\end{align}
and, together with \eqref{eq:3.21}, \eqref{eq:3.26} and \eqref{eq:3.28}, we obtain
	\begin{align}\label{eq:3.31}
		\int_{-\infty}^\infty  \| \mr H(\sigma + i \omega) \|_X \dd \omega  \leq  \kappa^2 \frac{e_{m,c}} {\ul \sigma^m} 2^{m/2} \frac \pi \sigma \int_0^\infty   \| (\mc P_m g^{(m+4)}) (\tau) \|_X \dd \tau.
	\end{align}
Here, since \eqref{eq:3.31} is true for any $c \in (0,\pi)$, we can replace $e_{m,c}$ with $e_m = \ds \min_{c \in (0, \pi)} e_{m,c} $ which finishes the proof.
\end{proof}

\begin{lemma}\label{lem:5.2}
Let $m \geq 0$, $g \in W_+^{2m+4}(\R; X)$, $h := e^\punto g^{(m+4)}$, and $t > 0$ be a fixed real number. We define the function
	\begin{equation}\label{eq:3.14}
		\R \ni \omega \longmapsto  j(\omega) := e^{-\omega} \ds \sum_{\ell = 0}^{m-1} \frac{(\omega - t)^\ell}{\ell !} h^{(\ell)}(t),
	\end{equation}
and, for $n \geq 1$, the integration operator
	\[
		(\partial^{-n} f)(\tau) := \int_t^\tau \int_t^{\omega_1} \cdots \int_t^{\omega_{n-1}} f(\omega_n) \, \dd \omega_n \, \cdots \,\dd \omega_2 \,\dd \omega_1.
	\]
The function
	\[
		p(\tau) := \left\{
		\begin{array}{ll}
		g(\tau), 	& 	\tau \leq t,\\[0.5em]
		\ds \sum_{\ell = 0}^{m+3} \frac{(\tau - t)^\ell}{\ell !} g^{(\ell)}(t)
			 + (\partial^{-m-4} j) (\tau), 	& 	\tau \geq t,
		\end{array}\right.
	\]
satisfies that $p \in W_+^{2m+4} (\R;X)$ and $\mc P_m p^{(m+4)} \equiv 0$ on $(t, \infty)$.
\end{lemma}

\begin{proof}
We start with observing that the function
	\begin{align*} 
		( \partial^k \, \partial^{-m-4} \, j )(\tau) = (\partial^{-m-4 + k} \, j )(\tau)
	\end{align*}
vanishes at $\tau = t$ for $0 \leq k \leq m+3$, because of the definition of the integration operator in \eqref{eq:3.14}. From there, it is not hard to see that $p \in \mc C^{m+3}(\R; X)$.
Next, using $h$ in the lemma, we define the function
	\begin{align*} 
		q(\tau) := \left\{
		\begin{array}{ll}
		h(\tau), 	& 	\tau \leq t,\\[0.5em]
		\ds \sum_{\ell = 0}^{m-1} \frac{(\tau - t)^\ell}{\ell !} h^{(\ell)}(t), 	& 	\tau \geq t.
		\end{array}\right.
	\end{align*}
Since $h \in \mc C^{m-1}(\R;X)$, $q$ is also in $\mc C^{m-1}(\R;X)$, and so is $e^{-\punto} q$. Therefore, we know that
	\begin{align*} 
		e^{-\tau} q(\tau)
		 = \left\{
		\begin{array}{ll}
		g^{(m+4)}(\tau), 	& 	\tau \leq t,\\[0.5em]
		 e^{-\tau} \ds \sum_{\ell = 0}^{m-1} \frac{(\tau - t)^\ell}{\ell !} h^{(\ell)}(t), 	& 	\tau \geq t.
		\end{array}\right\}
		 = 
		\frac{\mr d^{m+4}p}{\mr d \tau^{m+4}} (\tau) 
	\end{align*}
is in $\mc C^{m-1}(\R;X)$, which shows that $p \in \mc C^{2m+3} (\R;X)$. Next, we write
	\[
		p^{(2m+4)}(\tau) = \left\{
		\begin{array}{ll}
		g^{(2m+4)}(\tau), 	& 	\tau \leq t,\\[0.5em]
		j^{(m)} (\tau), 	& 	\tau \geq t.
		\end{array}\right.
	\]
Here $j^{(m)}(\tau)$ takes the form $j^{(m)}(\tau) = e^{-\tau} r(\tau)$, for a polynomial $r$, which implies that $j^{(m)} \in L^1(t,\infty;X)$. This, together with the fact that $g^{(2m+4)} \in L^1(\R;X)$, shows that $p^{(2m+4)} \in L^1(\R;X)$ and hence in $W^{2m+4}_+(\R;X)$. The rest of the proof follows from the fact that
	\begin{align*}
		( \mc P_m p^{(m+4)} )(\tau)  & = e^{-\tau} ( e^\punto p^{(m+4)} )^{(m)}(\tau) \\
		& = 
		e^{-\tau} \frac{\mr d^m}{\mr d\tau^m} \left(  \sum_{\ell = 0}^{m-1} \frac{(\tau - t)^\ell}{\ell !} h^{(\ell)}(t)  \right) (\tau) = 0 \qquad \forall \tau \in (t,\infty).
	\end{align*}
\end{proof}

\begin{proposition}\label{prop:3.5}
Let $g \in W_+^{2m+4}(\R; X)$ with $m \geq 0$, and $\mr F$ be as in \eqref{eq:1.1}. We define $f := \mc L^{-1} \{ \mr F \}$, and $ (\partial_t^\kappa)^m g$ such that
\[
\mc L \{ (\partial_t^\kappa)^m g \} (s) = s^m_\kappa \mr G(s), \qquad \mr G(s) := \mc L \{ g \} (s),
\]
for all $s \in \C_+$. The following estimate holds for all $t > 0$
\[
 \| f * (  (\partial_t^\kappa)^m g - g^{(m)}  ) (t) \|_Y
 \leq 
 \kappa^2 C_1(t^{-1}) \sup_{\Re s = t^{-1}}  \| \mr F (s) \| \int_0^t  \| \mc P_m g^{(m+4)} (\tau) \|_X \mr d \tau,
\]
where
\[
C_1(\sigma) := \frac{C_m}{\sigma \ul \sigma^m},
\]
and $C_m := \dfrac{e}{2 \pi} C^1_m$ is a positive constant depending only on $m$ with $C^1_m$ as defined in \eqref{eq:3.18}.
\end{proposition}
\begin{proof}
We start with fixing $t > 0$. Next, for any $\sigma > 0$, using the inverse Laplace transformation, we write
\[
 \| f * (  (\partial_t^\kappa)^m g - g^{(m)}  ) (t) \|_Y
 \leq 
\frac{e^{\sigma t}}{2 \pi} \sup_{\Re s = \sigma }  \| \mr F (s) \| \int_{-\infty}^{\infty}  \| \mr H(\sigma + i \omega) \|_X \mr d \omega,
\]
where
$\mr H(s) := s^m_\kappa \mr G(s) - s^m \mr G(s)$. Here, with the help of the Proposition \ref{prop:3.3} and inserting $\sigma = t^{-1}$, we obtain
\begin{equation}\label{eq:3.15}
 \| f * (  (\partial_t^\kappa)^m g - g^{(m)}  ) (t) \|_Y
 \leq 
 \kappa^2 C_1(t^{-1}) \sup_{\Re s = t^{-1}}  \| \mr F (s) \| \int_0^\infty  \| \mc P_m g^{(m+4)} (\tau) \|_X \mr d \tau.
\end{equation}
Now, our goal is to have an integral bound over the interval $(0,t)$. To do that, we consider the function $p$ introduced in Lemma \ref{lem:5.2}. Since $p \in W^{2m+4}_+(\R;X)$, in other words, it satisfies the conditions of this proposition, we can have the estimate \eqref{eq:3.15} for $p$ as well. Therefore, using the properties of this function, we write
\begin{align*}
 \| f * (  (\partial_t^\kappa)^m g - g^{(m)}  ) (t) \|_Y 
 & =  
 \| f * (  (\partial_t^\kappa)^m p - p^{(m)}  ) (t) \|_Y\\
 & \leq
  \kappa^2 C_1(t^{-1}) \sup_{\Re s = t^{-1}}  \| \mr F (s) \| \int_0^\infty  \| \mc P_m p^{(m+4)} (\tau) \|_X \mr d \tau \\
 & =
  \kappa^2 C_1(t^{-1}) \sup_{\Re s = t^{-1}}  \| \mr F (s) \| \int_0^t  \| \mc P_m g^{(m+4)} (\tau) \|_X \mr d \tau.
\end{align*}
This finishes the proof.
\end{proof}

\section{Laplace Domain estimates when $\mu \leq 0$}
In this section we establish some results which will be useful for the error and stability estimates of TRCQ method. We assume that $\mr F$ is as in \eqref{eq:1.1} with $\mu \leq 0$.
\begin{proposition}[Pointwise stability]  \label{prop:4.1}
For all $s \in \C_+$ we have
	\begin{align*}
		 \| \mr F_\kappa (s) \| \leq \Theta_1 (\Re s),
	\end{align*}
where 
	\begin{align} \label{eq:4.3}
		\Theta_1 (\sigma) := C_F \left( \tfrac12 \ul \sigma \right)  \left( \tfrac12 \ul \sigma \right)^\mu
	\end{align}
is non-increasing on $(0,\infty)$ and bounded by an inverse power of $\sigma$ at $0$.
\end{proposition}

\begin{proof}
Since $\Re s_\kappa > 0$ (see Proposition \ref{prop:2.3}(a)), \eqref{eq:1.1} implies
	\begin{align*}
		 \| \mr F(s_\kappa) \| \leq C_F(\Re s_\kappa) \, |s_\kappa|^\mu.
	\end{align*}
Using Proposition \ref{prop:2.3}(a) and the fact that $C_F$ is non-decreasing, we write
	\begin{align*}
		C_F(\Re s_\kappa)  \leq C_F \left( \tfrac12 \min\{ 1, \Re s \}  \right).
	\end{align*}
Next, using $\mu \leq 0$, we obtain
	\begin{align*}
		|s_\kappa|^\mu \leq  | \Re s_\kappa |^\mu \leq  \left( \tfrac12 \min\{ 1, \Re s \}  \right)^\mu,
	\end{align*}
which finishes the proof.
\end{proof}

\begin{proposition}[Pointwise approximation]\label{prop:4.2}
The following holds
\begin{enumerate}
\item[{\rm (a)}] $ \| \mr F'(s) \|  \leq \Theta_2 (\Re s) |s|^\mu$ for all $s \in \C_+$.
\item[{\rm (b)}] $ \| \mr F_\kappa (s) - \mr F(s)  \|  \leq \Theta_2 \left(\tfrac12 \min\{1,\Re s\}\right) \Theta_3 (|\kappa s|) \kappa^2 |s|^{\mu +3}$ for all $s \in \C_+$ with $|\kappa s| < c_0$.
\end{enumerate}
Here
	\begin{align}
		\Theta_2(\sigma) & := \frac{2^{1-\mu}}\sigma  C_F(\tfrac12 \sigma), \label{eq:4.7a}
	\end{align}
defined on $(0,\infty)$, is non-increasing, whereas
	\begin{align*}
		\Theta_3 (\sigma) & := D(\sigma) (1 - \sigma^2 D(\sigma))^\mu,
	\end{align*}
defined on $(0,c_0)$, is increasing and it diverges as $\sigma \to c_0$. Function $D$ and constant $c_0$ are as in Proposition \ref{prop:2.2}.
\end{proposition}

\begin{proof}
To prove (a), we use the ideas presented in the proof of \cite[Proposition 4.5.3]{Sayas2016}. We start with defining the curve
	\begin{align*}
		\Xi (s) := \left\{ z \in \C : |z - s|  = \frac {\Re s} 2 \right\},
	\end{align*}
with positive orientation. Using the Cauchy integral formula
	\begin{align*}
		\mr F'(s) 
		& = \frac1{2\pi i} \int_{\Xi(s)} \frac {\mr F(z)}{ (z-s)^2 } \dd z,
	\end{align*}
and $| \Xi(s) | = \pi \Re s$, we have
	\begin{align} \label{eq:4.10}
		 \| \mr F'(s) \|
		& \leq \frac1{2\pi} | \Xi(s) | \frac 1 {(\Re s/2)^2} \max_{z \in \Xi(s)}  \| \mr F(z) \|
		= \frac2{\Re s}  \max_{z \in \Xi(s)}  \| \mr F(z) \|.
	\end{align}
Next, for every $z \in \Xi(s)$ we write
	\begin{align*}
		|s| - |z| \leq |s - z| = \frac {\Re s} 2 \leq \frac {|s|} 2, \qquad  \Re s - \Re z \leq |s - z| = \frac {\Re s} 2 
	\end{align*}
which implies 
	\begin{align*}
		\frac {|s|} 2 \leq |z|, \qquad \frac{\Re s}{2} \leq \Re z.
	\end{align*}
Therefore, using the fact that $C_F$ and $(\cdot)^\mu$ are non-increasing on $(0,\infty)$, we can write
	\begin{align*}
		 \| \mr F(z) \| \leq C_F(\Re z) |z|^\mu \leq C_F(\tfrac12 \Re s) \frac{|s|^\mu}{2^\mu}.
	\end{align*}
This, together with \eqref{eq:4.10}, finishes the proof of (a). 

To prove (b), using the Mean Value Theorem, we write
	\begin{align}\label{eq:4.14}
		 \| \mr F_\kappa (s) - \mr F(s)  \|  \leq  \| \mr F'( \lambda s_\kappa + (1-\lambda) s) \| \, |s_\kappa - s|,
	\end{align}
for some $\lambda \in (0,1)$. Now, we first work on the first part of the right-hand side, and using (a), write the following 
	\begin{align}\label{eq:4.15}
		 \| \mr F'( \lambda s_\kappa + (1-\lambda) s) \| 
		 \leq 
		 \Theta_2  \left(  \Re ( \lambda s_\kappa + (1-\lambda) s ) \right)
		  \left| \lambda s_\kappa + (1-\lambda) s \right|^\mu.
	\end{align}
Here, Proposition \ref{prop:2.3}(a) implies
	\begin{align}\label{eq:4.16}
		\Re  \left( \lambda s_\kappa + (1- \lambda)s \right) 
		& \geq \min \{ \Re s_\kappa, \Re s \} \nonumber \\
		& \geq \min \left\{ \frac { \min \{1,\Re s\} }2, \Re s \right\} =\frac { \min \{1,\Re s\} }2.
	\end{align}
Next, using Proposition \ref{prop:2.3}(d), for every $s \in \C_+$ with $|\kappa s| < c_0$ we write
	\begin{align}\label{eq:4.17}
		|\lambda s_\kappa + (1-\lambda) s| 
		& = |s| \, \left|\lambda \frac{s_\kappa}{s} + (1-\lambda)\right| \nonumber\\
		& \geq |s| \, \Re \left(\lambda \frac{s_\kappa}{s} + (1-\lambda) \right) \nonumber\\
		& \geq |s| \, \min \left\{ \Re \frac {s_\kappa} s ,1 \right\} \geq |s|  \left( 1 - |\kappa s|^2 D(|\kappa s|) \right).
	\end{align}
Therefore, combining \eqref{eq:4.15}-\eqref{eq:4.17} with the fact that $\Theta_2$ and $(\cdot)^\mu$ are non-increasing on $(0,\infty)$, we obtain
	\begin{align} \label{eq:4.18}
		 \| \mr F'( \lambda s_\kappa + (1-\lambda) s) \| \leq \Theta_2 \left(\tfrac12 \min\{1,\Re s\}\right)\left( 1 - |\kappa s|^2 D(|\kappa s|) \right)^\mu |s|^\mu.
	\end{align}
Now, we bound the second part of the right-hand side of \eqref{eq:4.14} with the help of Proposition \ref{prop:2.3}(c) (and that $E_1 = D$)
	\begin{align*}
		|s_\kappa - s| \leq D(|\kappa s|) \kappa^2 |s|^3 \qquad |\kappa s| < \pi.
	\end{align*}
This, together with \eqref{eq:4.14} and \eqref{eq:4.18}, finishes the proof.
\end{proof}
Next result will help us to estimate integrals with inverse powers of $|s|$.
\begin{lemma}\label{lem:4.3}
The following holds for all $\sigma > 0$, $\alpha > 1$ and $c > 0$:
\begin{enumerate}[(a)]
\item[{\rm (a)}] 
$\ds \int_{|\sigma + i \omega| \geq c/\kappa} |\sigma + i \omega|^{-\alpha} \, \dd \omega 
\leq 
\frac{2 \alpha}{\alpha - 1}  \left( \frac{\kappa}{c} \right)^{\alpha-1}.$
\item[{\rm (b)}] 
$\ds \int_{-\infty}^\infty |\sigma + i \omega|^{-\alpha} \, \dd \omega 
\leq 
\frac{2}{\sigma^\alpha} + \frac{2}{\alpha - 1}.$
\end{enumerate}
\end{lemma}
\begin{proof}
We will estimate the integral in (a) separately depending on whether $\omega$ is close to $0$ or not. To do that, for fixed $\sigma, \kappa$ and $c$, we define the following domains of integration
	\begin{align*}
		I^1 &:=  \left\{ \omega \in \R: |\sigma + i \omega| \geq c/\kappa, |\omega| \leq c / \kappa  \right\}, \\
		I^2 & :=  \left\{ \omega \in \R: |\sigma + i \omega| \geq c/\kappa, |\omega| \geq c / \kappa  \right\} = \{ \omega \in \R: |\omega| \geq c/\kappa \},
	\end{align*} 
and write
	\begin{align*}
		\int_{|\sigma + i \omega| \geq c/\kappa} |\sigma + i \omega|^{-\alpha} \, \dd \omega
		=
		\int_{I^1} |\sigma + i \omega|^{-\alpha} \, \dd \omega
		+
		\int_{I^2} |\sigma + i \omega|^{-\alpha} \, \dd \omega.
	\end{align*}
First integral can be estimated using the fact that $\alpha > 0$, $|I^1| \leq \dfrac {2c} \kappa$, and $|\sigma + i \omega| \geq \dfrac c \kappa$ on $I^1$ in the following way
	\begin{align} \label{eq:4.23}
		\int_{I^1}  |\sigma + i \omega|^{-\alpha} \, \dd \omega \leq |I^1|  \left( \frac{c}{\kappa} \right)^{-\alpha} \leq 2  \left( \frac{\kappa}{c} \right)^{\alpha - 1}.
	\end{align}
Next, we work on the second integral by introducing the change of variables $\tilde \omega = \omega \kappa$
	\begin{align*}
		\int_{I^2} |\sigma + i \omega|^{-\alpha} \, \dd \omega 
		& =
		\kappa^{\alpha - 1} \int_{|\omega| \geq c} |\sigma \kappa + i \omega|^{-\alpha} \, \dd \omega \nonumber\\
		& \leq 
		\kappa^{\alpha - 1} \int_{|\omega| \geq c} | \omega|^{-\alpha} \, \dd \omega
		=
		2 \kappa^{\alpha - 1} \int_c^\infty  \omega^{-\alpha} \, \dd \omega.
	\end{align*}
Here, since $\alpha > 1$, we can compute the following integral exactly
	\begin{align*}
		 \int_c^\infty  \omega^{-\alpha} \, \dd \omega = \frac {c^{1-\alpha}}  {\alpha - 1} ,
	\end{align*}
and therefore we obtain
	\begin{align*}
		\int_{I^2} |\sigma + i \omega|^{-\alpha} \, \dd \omega 
		\leq
		\frac{2}{\alpha - 1} \left( \frac{\kappa}{c} \right)^{\alpha - 1}.
	\end{align*}
Combining this with \eqref{eq:4.23}, we finish the proof of (a). To show (b), using the fact that $|\sigma + i \omega| \geq |\sigma|, |\omega|$ we write
\begin{subequations}
	\begin{align}
		 \int_{-\infty}^\infty |\sigma + i \omega|^{-\alpha} \, \dd \omega 
		 & \leq 
		 2 \int_0^1 \sigma^{-\alpha} \, \dd \omega + 2 \int_1^\infty \omega^{-\alpha} \, \dd \omega  \label{eq:4.27a}\\
		 & =
		 \frac2 {\sigma^\alpha} + \frac2{\alpha-1}, \label{eq:4.27b}
	\end{align}
where we computed the second integral on the right-hand side of \eqref{eq:4.27a} using the fact that $\alpha > 1$. This finishes the proof.
\end{subequations}
\end{proof}
Next result uses the pointwise stability and approximation results, and is the key ingredient in the proof of Proposition \ref{prop:5.1}.
\begin{proposition}\label{prop:4.4}
For any $-3<\mu\leq0$, $\alpha > \mu + 4$, $0 < c < c_0 < \pi$ and $\sigma > 0$, the following holds
	\begin{align}\label{eq:4.28}
		\int_{-\infty}^\infty  \| \mr F_\kappa (\sigma + i \omega) - \mr F(\sigma + i \omega) \| \, |\sigma + i \omega|^{-\alpha} \, \dd \omega 
		\leq 
		\kappa^2 e_1(\sigma) + \kappa^{\alpha - 1} e_2(\sigma),
	\end{align}
where
\begin{subequations}\label{eq:4.29}
	\begin{align}
		e_1(\sigma) & := \Theta_2 (\tfrac12 \ul \sigma ) \Theta_3 (c)  \left( \frac2 {\sigma^{\alpha - \mu -3}} + \frac 2 {\alpha - \mu - 4} \right),\\
		e_2(\sigma)& := \frac{2 \alpha}{\alpha - 1} \frac1 {c^{\alpha-1}} \left( \Theta_1 (\sigma) + C_F(\sigma) \sigma^\mu \right).
	\end{align}
\end{subequations}
\end{proposition}

\begin{proof}
For fixed $\sigma, \kappa$ and $c$,  we will make use of the following domains of integration
	\begin{align*}
		I^1 :=  \left\{ \omega \in \R : | \sigma + i \omega | \leq \frac c  \kappa \right\}, \qquad 
		I^2 :=  \left\{ \omega \in \R : | \sigma + i \omega | \geq \frac c  \kappa \right\},
	\end{align*}
and the notation
	\begin{align*}
		s = s(\omega) := \sigma + i \omega.
	\end{align*}
We start with splitting the integral in \eqref{eq:4.28} in the following way
	\begin{align}\label{eq:4.34}
		\int_{-\infty}^\infty  \| \mr F_\kappa (s) - \mr F(s) \| \, |s|^{-\alpha} \, \dd \omega
		& \leq
		\int_{I^1}  \| \mr F_\kappa (s) - \mr F(s) \| \, |s|^{-\alpha} \, \dd \omega \nonumber \\
		& \quad + \int_{I^2}  \left(  \| \mr F_\kappa (s) \| +  \| \mr F(s) \|   \right) \, |s|^{-\alpha} \, \dd \omega.
	\end{align}
For the first integral, we use Proposition \ref{prop:4.2}(b) and the fact that $\Theta_3$ is increasing to write
	\begin{align} \label{eq:4.35}
		\int_{I^1}  \| \mr F_\kappa (s) - \mr F(s) \| \, |s|^{-\alpha} \, \dd \omega
		\leq 
		\kappa^2 \Theta_2 (\tfrac12 \ul \sigma)  \int_{I^1}  \Theta_3 (|\kappa s|)   \, |s|^{-\alpha + \mu + 3} \, \dd \omega \nonumber \\
		\leq 
		\kappa^2 \Theta_2 (\tfrac12 \ul \sigma) \Theta_3 (c) \int_{-\infty}^\infty  |s|^{-(\alpha - \mu - 3)} \, \dd \omega.
	\end{align}
Here since $\alpha - \mu -3 >1$, Lemma \ref{lem:4.3}(b) implies
	\begin{align*}
		\int_{-\infty}^\infty  |s|^{-(\alpha - \mu - 3)} \, \dd \omega
		\leq
		\frac2 {\sigma^{\alpha - \mu -3}} + \frac 2 {\alpha - \mu - 4},
	\end{align*}
which, together with \eqref{eq:4.35}, gives us the first half of the bound in  \eqref{eq:4.28}
	\begin{align} \label{eq:4.37}
		\int_{I^1}  \| \mr F_\kappa (s) - \mr F(s) \| \, |s|^{-\alpha} \, \dd \omega \leq \kappa^2 e_1(\sigma).
	\end{align}
Now we bound the other integral in \eqref{eq:4.34} using Proposition \ref{prop:4.1} and \eqref{eq:1.1}
	\begin{align} \label{eq:4.38}
		\int_{I^2} \left( \| \mr F_\kappa (s) \| + \| \mr F(s) \| \right) \, |s|^{-\alpha} \, \dd \omega 
		& \leq 
		\left( \Theta_1 (\sigma) + C_F(\sigma) \sigma^\mu \right) \int_{I^2} |s|^{-\alpha} \, \dd \omega \nonumber \\
		& \leq
		\left( \Theta_1 (\sigma) + C_F(\sigma) \sigma^\mu \right)\frac{2 \alpha}{\alpha - 1}  \left( \frac{\kappa}{c} \right)^{\alpha-1} \!\!\! = \kappa^{\alpha - 1} e_2(\sigma),
	\end{align}
where for the last line we used Lemma \ref{lem:4.3}(a). This bound and \eqref{eq:4.37} finishes the proof. 
\end{proof}
\begin{corollary}\label{cor:4.5}
If $-1 < \mu \leq 0$ and $\alpha = \lfloor \mu + 5 \rfloor$, then for all $\sigma > 0$ we have
	\begin{align*}
		\int_{-\infty}^\infty  \| \mr F_\kappa (\sigma + i \omega) - \mr F(\sigma + i \omega) \| \, |\sigma + i \omega|^{-\alpha} \, \dd \omega 
		\leq 
		\kappa^2 \, C_\mu^1 \, C_2(\sigma),
	\end{align*}
where
\begin{equation} \label{eq:4.39}
C_2(\sigma) :=  C_F(\tfrac14 \ul\sigma) \frac {1}{\ul \sigma^{2 + \delta}},
\qquad 
\delta := \floor\mu - \mu + 1 \in (0,1],
\end{equation}
and $C_{\mu}^1$ is a positive constant depending only on $\mu$.
\end{corollary}
\begin{proof}
We first note that
	\begin{align*}
		\kappa^{\alpha -1} \leq \kappa^2,
	\end{align*}
since $\alpha - 1 > \mu + 3 > 2$. Next, recalling the definitions of $\Theta_1$ and $\Theta_2$ (See \eqref{eq:4.3} and \eqref{eq:4.7a}) we write
	\begin{align*}
		\Theta_1 (\sigma) \leq C_F(\tfrac12 \ul \sigma) \frac2 {\ul \sigma},\qquad
		\Theta_2 (\tfrac12 \ul \sigma) \leq \frac{2^{2-\mu}}{\ul \sigma} C_F(\tfrac 14 \ul \sigma).
	\end{align*}
Using these, we bound $e_1$ and $e_2$ (see \eqref{eq:4.29} for their definition) in the following way
	\begin{align*}
		e_1(\sigma) 
		 \leq
		\frac1{\ul \sigma^{\alpha - \mu -2}} C_F(\tfrac 14 \ul \sigma) e^1_\mu\Theta_3 (c) ,
	\qquad
		e_2(\sigma)
		\leq
		\frac1{\ul \sigma} C_F(\tfrac12 \ul\sigma) e^2_\mu \, c^{1-\alpha},
	\end{align*}
with
	\begin{align*}
		e^1_\mu := 2^{3-\mu} \max \left\{ 1, \frac{1}{\alpha - \mu - 4} \right\}, \qquad 
		e^2_\mu :=  \frac{8 \alpha}{\alpha - 1}.
	\end{align*}
Hence, we obtain
	\begin{align*}
		e_1(\sigma) + e_2(\sigma) \leq \frac1{\ul \sigma^{2 + \delta}} C_F(\tfrac 14 \ul \sigma) (e^1_\mu \, \Theta_3 (c) + e^2_\mu \, c^{1-\alpha}),
	\end{align*}
where we used the fact that $\delta = \alpha - \mu - 4$. Note that this inequality holds for any $c \in (0,c_0)$, and we know that the function 
\[
(0,c_0) \ni c \longmapsto e^1_\mu \, \Theta_3 (c) + e^2_\mu \, c^{1-\alpha}
\]
is continuous and bounded below, therefore it attains a minimum value. This implies that
	\begin{align*}
		e_1(\sigma) + e_2(\sigma) \leq \frac1{\ul \sigma^{2 + \delta}} C_F(\tfrac 14 \ul \sigma) \, C_\mu^1,
	\end{align*}
where $C_\mu^1 = \ds \min_{c \, \in (0,c_0)} (e^1_\mu \, \Theta_3 (c) + e^2_\mu \, c^{1-\alpha})$, which finishes the proof.
\end{proof}

\section{Convergence of TRCQ}
This section contains our main convergence result. In order to establish that, we will first talk about Lubich's result tailored for the case $-1 < \mu \leq 0$. One important comment here is that although the result below is based on \cite[Theorem 3.1]{Lubich1994}, it includes the case $\mu = 0$, which was not covered there. Moreover, we are interested in the analysis of TRCQ, whereas Lubich discusses results for CQ in a general way. Our proof makes use of shortcuts and includes more details about the hidden constants.

In this section we will have
	\begin{align} \label{eq:5.1}
		f := \mc L^{-1} \{\mr F \}, \qquad 
		f_\kappa := \mc L^{-1} \{\mr F_\kappa \}.
	\end{align} 

\begin{proposition} \label{prop:5.1}
If $\mr F$ satisfies \eqref{eq:1.1} with $-1<\mu \leq 0$, and $g \in W_+^{\alpha}(\R; X)$ with $\alpha := \lfloor \mu + 5 \rfloor $, then
	\begin{align}\label{eq:5.2}
		\| (f_\kappa - f) * g(t) \|_Y \leq \kappa^2 \, C_\mu^2 \, C_2(t^{-1}) \int_0^t \| g^{(\alpha)} (\tau) \|_X \, \dd \tau
	\end{align}
holds for all $t \geq 0$, where $C_2$ is as in \eqref{eq:4.39}, and $C_\mu^2$ is a constant depending only on $\mu$.
\end{proposition}
\begin{proof}
We start with defining $\mr G := \mc L \{g \}$. Using the inverse Laplace transformation, we have the following estimate for any $\sigma > 0$
	\begin{align}\label{eq:5.4}
		\|  (f_\kappa - f) * g(t) \|_Y 
		&\leq 
		\frac{e^{\sigma t}}{2\pi} \int_{-\infty}^\infty \| \mr F_\kappa(\sigma + i \omega) - \mr F(\sigma + i \omega)\| \, \|\mr G(\sigma + i \omega)\|_X \, \dd \omega \nonumber \\
		& \leq 
		\frac{e^{\sigma t}}{2\pi} \sup_{\Re s = \sigma} \| \mr G(s) s^\alpha \|_X  \!\!
				\int_{-\infty}^\infty \| (\mr F_\kappa- \mr F)(\sigma + i \omega)\| \, | \sigma + i \omega |^{-\alpha} \, \dd \omega.
	\end{align}
We know that for all $\Re s = \sigma$
	\begin{align*}
		\| \mr G(s) s^\alpha \|_X \leq e^{-\sigma t} \int_0^\infty  \| \mc L^{-1} \{\mr G(s) s^\alpha\}  (\tau) \|_X \, \dd \tau \leq \int_0^\infty \| g^{(\alpha)} (\tau) \|_X \, \dd \tau
	\end{align*}
satisfies. Next, using Corollary \ref{cor:4.5}, we obtain
	\begin{align*}
		\|  (f_\kappa - f) * g(t) \|_Y 
		\leq 
		\kappa^2 \frac{e^{\sigma t}}{2\pi} \, C_\mu^1 \, C_2(\sigma) \int_0^\infty \| g^{(\alpha)} (\tau) \|_X \, \dd \tau,
	\end{align*}
 and setting $\sigma = t^{-1}$ gives
	\begin{align}\label{eq:5.6}
		\|  (f_\kappa - f) * g(t) \|_Y 
		\leq  \kappa^2 \, C_\mu^2 \, C_2(t^{-1}) \int_0^\infty \| g^{(\alpha)} (\tau) \|_X \, \dd \tau,
	\end{align}
with $C^2_\mu = \dfrac {e}{2 \pi} C^1_\mu$.
We are very close to obtaining \eqref{eq:5.2}. Now, for a fixed $t >0$, we define the following function
	\begin{align}\label{eq:5.7}
		p(\tau) := \left\{
		\begin{array}{ll}
		g(\tau), & \tau \leq t,\\[0.5em]
		\ds \sum_{\ell = 0}^{\alpha-1} \frac{(\tau - t)^\ell}{\ell !} g^{(\ell)}(t), & \tau \geq t.
		\end{array}\right.
	\end{align}
It is not hard to see that $p \in W^{\alpha}_+(\R;X)$, in other words, satisfies the conditions of the proposition. We also know $p \equiv g$ on $(-\infty,t)$ and $p^{(\alpha)} \equiv 0$ on $(t,\infty)$. 
Therefore we have
	\begin{align}\label{eq:5.8}
		\|  (f_\kappa - f) * g(t) \|_Y = \|  (f_\kappa - f) * p(t) \|_Y
		& \leq 
		 \kappa^2 \, C_\mu^2 \, C_2(t^{-1}) \int_0^\infty \| p^{(\alpha)} (\tau) \|_X \, \dd \tau \nonumber \\
		& = 
		 \kappa^2 \, C_\mu^2 \, C_2(t^{-1}) \int_0^t \| g^{(\alpha)} (\tau) \|_X \, \dd \tau.
	\end{align}
\end{proof}

\begin{theorem}\label{thm:5.3}
Let $\mr F$ be as in \eqref{eq:1.1} with $\mu \geq 0$, 
	\begin{align*}
		m := \ceil \mu, \qquad \alpha := \floor {\mu -m} + 5, \qquad \beta := \max \{2m + 4, m + \alpha\},
	\end{align*}
and $g \in W_+^\beta(\R;X)$. The following holds for all $t>0$
	\begin{align*}
		\| (f_\kappa - f)*g(t) \|_Y 
		\leq
		\kappa^2 C(t^{-1}) \left( \int_0^t \| g^{(m+\alpha)}(\tau) \|_X \, \dd \tau + \int_0^t \| \mc P_m g^{(m+4)} (\tau) \|_X \, \dd \tau \right).
	\end{align*}
Here
	\begin{align*}
		C(\sigma) := C_F(\tfrac14 \ul \sigma) \frac{C_{\mu}}{\ul \sigma^{\varepsilon}},
	\end{align*}
where, $\varepsilon := \max \{2m -\mu + 1, \floor{\mu} - \mu + 3 \}$, and $C_{\mu}$ is a constant depending only on $\mu$.
\end{theorem}
\noindent\textbf{Note.} A few comments on $\mu,m,\alpha,\beta$ and $\varepsilon$:
\begin{itemize}
\item
$\mu -m \in (-1,0]$.
\item
$\alpha = \left\{\begin{array}{ll}
5, & \mu =m,\\
4,& \mu \neq m.
\end{array}\right.$
\item
$\beta = \left\{\begin{array}{ll}
5, & \mu = 0,\\
2m + 4,& \mu > 0.
\end{array}\right.$
\item
$\max\{2, m+1\}  \leq  \varepsilon  \leq  \max\{3, m+2\}$.
\end{itemize}

\begin{proof}
The following notation will be used:
\begin{subequations}
	\begin{alignat*}{6}
		& \mr F^m (s) := s^{-m} \mr F(s) = \mc L \{ f^m\}(s), \qquad && f^m = \partial^{-m}_t f,\\
		& \mr F^m_\kappa (s) := s_\kappa^{-m} \mr F(s_\kappa) = \mc L \{ f^m_\kappa\}(s), \qquad && f^m_\kappa = (\partial^\kappa_t)^{-m} f.
	\end{alignat*}
\end{subequations}
We start with writing
	\begin{align*}
		f * g = f^m * g^{(m)}, \qquad f _\kappa * g = f_\kappa^m * (\partial^\kappa_t)^{m}g,
	\end{align*}
and defining 
	\begin{align*}
		h := ((\partial^\kappa_t)^{m}g - g^{(m)}),
	\end{align*}
which gives us
	\begin{align} \label{eq:5.17}
		f_\kappa * g - f * g = (f_\kappa^m - f^m) * g^{(m)} + f_\kappa^m * h.
	\end{align}
Now, we are going to work on $(f_\kappa^m - f^m) * g^{(m)}$ and $f_\kappa^m * h$ separately.

For the first one, since we have 
	\begin{align}\label{eq:5.24}
		\| \mr F^m \| \leq C_F(\Re s) |s|^{\mu-m} \qquad \forall s \in \C_+,
	\end{align} 
with $-1< \mu-m \leq 0$, and $g^{(m)} \in W^{\alpha}_+(\R;X)$, we can apply Proposition \ref{prop:5.1} to obtain
	\begin{align}\label{eq:5.27}
		\| (f_\kappa^m - f^m) * g^{(m)} (t) \|_Y \leq  \kappa^2 \, C_{\mu-m}^2 \, C_2(t^{-1}) \int_0^t \| g^{(m+\alpha)} (\tau) \|_X \, \dd \tau,
	\end{align}
where $C^2_{\mu-m}$ corresponds to the constant $C^2_\mu$ in \eqref{eq:5.2}. Also recall that
	\begin{align*}
		C_2(\sigma) = C_F(\tfrac14 \ul\sigma) \frac {1}{\ul\sigma^{2 + \delta}},
	\end{align*}
with $\delta =  \floor{\mu - m} - (\mu - m) + 1 = \floor\mu - \mu + 1$.

Next, we bound $f_\kappa^m * h$ using Proposition \ref{prop:3.5} in the following way
\begin{equation}\label{eq:5.28}
 \| (f_\kappa^m * h) (t) \|_Y
 \leq 
 \kappa^2 C_1(t^{-1}) \sup_{\Re s = t^{-1}}  \| \mr F^m_\kappa (s) \| \int_0^t  \| \mc P_m g^{(m+4)} (\tau) \|_X \mr d \tau.
\end{equation}
Here, using the definition of $C_1$ together with Proposition \ref{prop:4.1}, we have
\begin{align*}
	C_1(\sigma) \sup_{\Re s = \sigma}  \| \mr F^m_\kappa (s) \|
	\leq
	\frac{C_m}{\sigma \ul \sigma^m}  C_F(\tfrac12 \ul \sigma) 2^{-\mu + m} \ul\sigma^{\mu-m}
	=
	C_F(\tfrac12 \ul \sigma) \frac{C^3_\mu}{ \sigma \ul \sigma^{2m - \mu}}
\end{align*}
for all $\sigma > 0$, where $C^3_\mu := C_m \, 2^{-\mu + m}$. Combining this with \eqref{eq:5.28}, and then \eqref{eq:5.27} we obtain
\[
\| (f_\kappa - f)*g(t) \|_Y 
		\leq
		C(t^{-1}) \left( \int_0^t \| g^{(m+\alpha)}(\tau) \|_X \, \dd \tau + \int_0^t \| \mc P_m g^{(m+4)} (\tau) \|_X \, \dd \tau \right).
\]
Here we used the fact that
\[
	\max \{ C_{\mu-m}^2 C_F(\tfrac14 \ul \sigma)  \frac1{\ul\sigma^{2+\delta}}, C_\mu^3 \, C_F(\tfrac12 \ul \sigma) \frac1{\sigma \ul \sigma^{2m - \mu}} \} 
		\leq 
		C_F(\tfrac14 \ul \sigma) \frac{C_{\mu} }{\ul \sigma^\varepsilon}
		=
		C(\sigma),
\]
where $C_\mu = \max\{ C^2_{\mu - m}, C^3_\mu \}$ and $\varepsilon = \max\{ 2 + \delta, 2m - \mu +1 \}$. This finishes the proof.
\end{proof}

\end{document}